\documentclass[12pt]{article}
\usepackage{lipsum}

\usepackage[centertags,intlimits]{amsmath}                     
\usepackage{color}\usepackage{amsfonts}                     
\usepackage{amsthm}
\usepackage{graphicx}
\allowdisplaybreaks[2]
\numberwithin{equation}{section}
\newtheorem{theorem}{Theorem}[section]
\newtheorem{lemma}{Lemma}[section]

\theoremstyle{remark}
\newtheorem{remark}{Remark}[section]

\providecommand{\abs}[1]{\lvert #1\rvert}

\newcommand{\nc}{\newcommand}
\nc{\vb}{\mathbf{v}}
\nc{\bx}{\mathbf{x}}
\nc{\by}{\mathbf{y}}
\nc{\bz}{\mathbf{z}}
\nc{\bu}{\mathbf{u}}
\nc{\bv}{\mathbf{v}}
\nc{\ba}{\mathbf{a}}
\nc{\bs}{\mathbf{s}}
\nc{\bq}{\mathbf{q}}
\nc{\bd}{\mathbf{d}}
\nc{\bb}{\mathbf{b}}
\nc{\bc}{\mathbf{c}}
\nc{\bi}{\mathbf{i}}
\nc{\bfr}{\mathbf{r}}
\nc{\bA}{\mathbf{A}}
\nc{\R}{\mathbb R}
\nc{\N}{\mathbb N}
\nc{\C}{\mathbb C}
\nc{\D}{\mathbb D}
\nc{\Z}{\mathbb Z}
\nc{\F}{\mathbf F}
\nc{\bbS}{\mathbb S}
\nc{\B}{\cal B}
\nc{\br}{\bigr}
\nc{\bl}{\bigl}
\nc{\Bl}{\Bigl}
\nc{\Br}{\Bigr}
\nc{\ind}{\mathbf{1}}
\nc{\bP}{\mathbf{P}}

\DeclareMathOperator*{\cl}{cl}  

\title{On the  metastability of a  loss network with diminishing rates}
\author{Anatolii A. Puhalskii }

\begin{document}
\maketitle
\sloppy
\vspace{1.cm}


\begin{abstract}
A trajectorial large deviation principle is established in a mean field
thermodynamic limit for a
multiclass loss
network with diminishing rates,
which may have several stable equilibria. The large deviation limit
 is identified as a solution to a maxingale
problem with a Markov property. 
 The
invariant measure of the network process   obeys a large deviation
principle as well. The network is 
 metastable in that it spends exponentially long periods
of time in the neighbourhoods of  stable equilibria. 
A specific case
of  a two--class network with two stable equilibria and one unstable equilibrium
is examined.   \end{abstract}
\begin{flushleft}
\mbox{MSC: 60F10, 60F17}
\end{flushleft}

\begin{flushleft}
\mbox{Key words:  large deviations; stochastic networks;
 metastability; invariant measures }
\end{flushleft}
\begin{flushleft}
\mbox{Short title: on the metastability}
\end{flushleft}

\section{Introduction}
\label{sec:introduction}

The following model of a cellular network was studied in
Antunes et al.  \cite{Rob08}. There are $n$ nodes of capacity $C$
each. Customers  of
$K$ classes   arrive  at the nodes according to
Poisson processes of respective rates $\alpha_k$\,, $1\le k\le K$\,. On arrival
at a node,
a class $k$ customer occupies $A_k$ units of the node's 
capacity, being rejected
and removed from the network if
  the required capacity
is  not available. On acceptance, the
customer stays at the node for an exponentially distributed length of time
with mean $1/\gamma_k$ and then moves to another node, a destination
 being chosen uniformly at random. As on arrival, 
rejection  occurs within the network
when less than
 $A_k$ units of unused capacity are available at a destination node.
A class $k$ customer may also leave
 the network after  an exponentially distributed 
length of time  of mean
$1/\delta_k$\,.
The arrival processes,
  sojourn times at the nodes,  sojourn times in
the network
 and routing decisions
are independent.   

Antunes et al.  \cite{Rob08}
obtained a law of large numbers for the process of the proportions of
nodes with a given population, as the number of nodes goes to
infinity. 
They analysed stability properties of
the limit dynamical system and showed that it may have several stable 
equilibria.
In Tibi \cite{Tib10}, it was argued that the network process
would spend exponentially
long periods of time in the neighbourhoods of stable equilibria, in
analogy with the developments in Freidlin and Wentzell \cite{wf2},
implying that the network is metastable.
Unfortunately, the analysis in Tibi \cite{Tib10} is not complete. 

As observed in Tibi \cite{Tib10}, this model stands out because
mean field behaviour arises in the limit only, 
the interactions within the network  being local.
 Usually, when
mean--field models are considered, the mean-field
interaction is built in the hypotheses.
In a similar vein, in the available literature
multistability of a  dynamical system, resulting in
metastability,   is assumed extraneously, for the most part,
  whereas in this model
it is an intrinsic feature, too.

To elucidate the contribution of this paper, one needs to
put things in a precise setting and review  the results in
Antunes et al. \cite{Rob08} in more detail.
The state of node $i$ at time $t$ is described by the vector
$X^{(n)}_i(t)=(X^{(n)}_{i,1}(t),\ldots, X^{(n)}_{i,K}(t))$\,, whose $k$-th
entry records the number of class $k$ customers at the node.
The process $X^{(n)}_i=(X^{(n)}_i(t)\,, t\ge0)$ takes values in the set
 $\Theta=\{\theta=(\theta_1,\ldots,\theta_K)\in\Z_+^K:\,
\sum_{k=1}^K\theta_kA_k\le C\}$\,. It is assumed that
$\abs{\Theta}\ge2$\,, $\abs{\Theta}$ denoting the cardinality of
$\Theta$\,.
Let $Y^{(n)}_\theta(t)$ represent the proportion of nodes with
$\theta$ as the population vector, i.e.,
\begin{equation*}
   Y^{(n)}_\theta(t)=\frac{1}{n}\,\sum_{i=1}^{n}\ind_{\{X^{(n)}_i(t)=\theta\}}
\end{equation*}
and  let $Y^{(n)}(t)=(Y^{(n)}_\theta(t)\,,\theta\in \Theta)$\,, where
$\ind_{\Xi}$ denotes the indicator of event $\Xi$\,.
As \begin{equation*}
    \sum_{\theta\in \Theta}Y^{(n)}_\theta(t)=1\,,
\end{equation*}
the  process $Y^{(n)}=(Y^{(n)}(t)\,, t\ge0)$ is a Markov process with
 values
 in the discrete simplex 
${\mathbb S}_{\abs{\Theta}}^{(n)}=\{y=(y_\theta,\,\theta\in \Theta):\,\sum_{\theta\in \Theta}y_\theta=1\,, y_\theta\ge0\,, ny_\theta\text{ is
  an integer}\}$\,.
Let  ${\mathbb S}_{\abs{\Theta}}=\{y=(y_\theta\,,\theta\in \Theta):\,\sum_{\theta\in \Theta}y_\theta=1\,, y_\theta\ge0\}$\,.
It follows from the results in Antunes et al. \cite{Rob08}
 that if the sequence $Y^{(n)}(0)$ converges
 in probability to $\hat y\in \mathbb
 S_{\abs{\Theta}}$\,, as $n\to\infty$\,, then the sequence 
$Y^{(n)}$ converges in
probability  uniformly over compact  intervals to the solution
$\by=(\by(t)\,, t\ge0)$ of the initial value problem
\begin{equation}
  \label{eq:27}
  \dot\by_\theta(t)=V_\theta(\by(t))
\end{equation}
and $\by(0)=\hat y\,,$
where
 $\by(t)=(\by_\theta(t)\,,\theta\in\Theta)\in\mathbb S_{\abs{\Theta}}$
and, for $y=(y_\theta\,,\theta\in\Theta)\in\mathbb S_{\abs{\Theta}}$\,, 
\begin{multline*}
    V_\theta(y)=\sum_{k=1}^K\bl((\alpha_k+\sum_{\theta'\in\Theta}\theta'_k\gamma_ky_{\theta'})y_{\theta-e_k}+(\delta_k+\gamma_k)(\theta_k+1)y_{\theta+e_k}\\-
(\alpha_k+(\delta_k+\gamma_k)\theta_k+\sum_{\theta'\in\Theta}\theta'_k\gamma_ky_{\theta'})y_\theta\br)\,,
\end{multline*}
with an  overdot denoting a time derivative,
 $e_k$ denoting the $k$th  vector of the
canonical basis of $\R^K$ and with the convention that
 $y_{\theta\pm e_k}=0$ if $\theta\pm e_k\notin\Theta$\,.

 Both 
$Y^{(n)}(t)$ and  $\by(t)$ are probability distributions on $\Theta$\,.
The equilibrium points of \eqref{eq:27} are given by an Erlang formula
for the stationary distribution of an $M/M/C/C$ queue.
More specifically, for $\rho=(\rho_1,\ldots,\rho_K)\in\R_+^K$\,,
let a probability distribution
$\nu(\rho)=(\nu_\theta(\rho)\,,\theta\in\Theta)$
on $\Theta$ be defined as
\begin{equation}
  \label{eq:48}
  \nu_\theta(\rho)=\frac{1}{Z(\rho)}\,
\prod_{k=1}^K\frac{\rho_k^{\theta_k}}{\theta_k!}\,,
\end{equation}
 $Z(\rho)$ being a normalising
constant. If, for $k=1,2,\ldots,K$\,,
\begin{equation}
  \label{eq:53}
  \rho_k=\frac{\alpha_k+\gamma_k\sum_{\theta\in\Theta}\theta_k\nu_\theta(\rho)}{\gamma_k+\delta_k}\,,
\end{equation}
then  $y=\nu(\rho)$ is an equilibrium of \eqref{eq:27}. Every
equilibrium is of this form. 
The existence of  solutions to \eqref{eq:48} and \eqref{eq:53}  is proved via an application
of  Brouwer's fixed point theorem.    
On the other hand,  uniqueness
of an equilibrium  for \eqref{eq:27}
might not hold and in Antunes et al. \cite{Rob08}  an example of a network with no less than two  stable
equilibria is provided, so, metastability is likely to occur.

  An essential
 stepping stone 
 toward proving  metastability is to derive a trajectorial large
 deviation principle (LDP) for the sequence of
 $Y^{(n)}$
as 
  random elements of the associated Skorohod space.
 General results on large
 deviations of Markov processes in Freidlin and Wentzell \cite{wf2}
 and
in Wentzell \cite{Wen90} fall short. A major sticking point is
what is known as the phenomenon of
  ''diminishing rates'', see Shwartz and Weiss \cite{SchWei05}: 
 near the
boundary of the state space the normalised  transition rates get vanishingly small, e.g., 
the transitions $y\to
y+(f_{\theta+e_{k-1}}-f_\theta)/n$\,, which correspond to departures
of class $k$ customers from nodes with population vector $\theta$\,,
occur at the rate
$ny_{\theta}\theta_k\delta_k$\,, which, when
divided by $n$\,, tends to $0$ as
$y_\theta\to0$\,, where 
$f_\theta$ denotes the $\theta$-th vector of the canonical basis of
 $\R^{\Theta}$\,. The line of attack in this paper is to
prove $\C$--exponential tightness of the sequence
of distributions of $Y^{(n)}$ 
  and to identify
a large deviation (LD) limit point as a solution to a maxingale problem,
cf., Puhalskii \cite{Puh01}. The issue of diminishing rates is
tackled by approximating trajectories that  reach the boundary of the
state space with trajectories that stay away from the boundary.
In the process, some new techniques are developed, e.g.,
the LD limit point is shown to have a Markov property which enables
one to identify it piecewise.

The trajectorial LDP is called upon, at first, in order  to obtain
 an LDP for the invariant measure of $Y^{(n)}$\,,
which is done by applying   the results in
Puhalskii \cite{Puh21}. Secondly,
following the developments in Freidlin
and Wentzell \cite{wf2} and in Shwartz and Weiss \cite{SchWei95},
 logarithmic asymptotics of both  exit
times   from the neighbourhoods of stable equilibria
and of the moments of the exit times are  obtained,
 thus establishing metastability.
 As an illustration, 
 a two--class
metastable  network is looked at,  which is similar to the one   in Antunes et
al. \cite{Rob08}. 

Here is how this paper is organised.
 The trajectorial LDP is stated and proved in
Section \ref{sec:proof-theor-refth}. Section
\ref{sec:large-devi-invar} is concerned with the LDP for the invariant
measure of  $Y^{(n)}$ and  the metastability. The paper uses
extensively the terminology and techniques of large deviation
convergence as expounded upon in Puhalskii \cite{Puh01}. A primer
is  available at the beginning of Section 3 in Puhalskii \cite{Puh21}.
\section{The trajectorial LDP }
\label{sec:proof-theor-refth}
 Let,
 for
 $y=(y_\theta\,,\theta\in\Theta)\in\R^\Theta,\,
z=(z_\theta\,,\theta\in\Theta)\in\R^\Theta,$ and
$\lambda=(\lambda_\theta\,,\theta\in\Theta)\in\R^\Theta$\,,
 \begin{multline}
   \label{eq:13}
   H(y,\lambda)=\sum_{k=1}^K\sum_{\theta\in\Theta_k^+}
(e^{\lambda_{\theta+e_k}-\lambda_{\theta}}-1)
\alpha_k y_\theta\\+\sum_{k=1}^K\sum_{\theta\in \Theta^-_k}
(e^{\lambda_{\theta-e_k}-\lambda_{\theta}}-1)
(\delta_k+
\gamma_k\sum_{\theta'\in\Theta\setminus\Theta^+_k}
y_{\theta'})\theta_ky_{\theta}
\\+\sum_{k=1}^K\sum_{\theta'\in \Theta^+_k,\,\theta\in\Theta^-_k}
(e^{\lambda_{\theta'+e_k}-\lambda_{\theta'}+\lambda_{\theta-e_k}-\lambda_{\theta}}-1)
\theta_k\gamma_k y_\theta y_{\theta'}
\end{multline}
and
\begin{equation}
  \label{eq:11}
  L(y,z)=\sup_{\lambda\in\R^\Theta}\bl(\lambda\cdot z-H(y,\lambda)\br)\,,
\end{equation}
where
$  \Theta^{\pm}_k=\{\theta\in \Theta:\,\theta\pm e_k\in\Theta\}$
and `` $\cdot$ '' is used to denote an inner product.
Let $\D(\R_+,\R^\Theta)$ denote the Skorohod space of right continuous
$\R^{\Theta}$--valued functions with lefthand limits. It is endowed
with a metric rendering it a complete separable metric space, see,
e.g., Ethier and Kurtz \cite{EthKur86}, Jacod and Shiryaev \cite{jacshir}.
 \begin{theorem}
\label{the:LDP}
Let $y^{(n)}\in\mathbb S^{(n)}_\Theta$\,, $y\in {\mathbb
  S}_{\abs{\Theta}}$\,, and
 $y^{(n)}\to y$ as $n\to\infty$\,. Then
the sequence $Y^{(n)}$ with $Y^{(n)}(0)=y^{(n)}$ obeys an LDP in $\D(\R_+,\R^\Theta)$  with deviation function
  \begin{equation}
    \label{eq:15}
    I_{  y}^\ast(\by)=\int_0^\infty L(\by(s),\dot\by(s))\,ds\,,
  \end{equation}
provided $\by=(\by(t)\,, t\ge0)$ is an absolutely continuous function
taking values in $\mathbb S_{\abs{\Theta}}$ with $\by(0)=y$\,, and $I^\ast_{ y}(\by)=\infty$\,, otherwise.
\end{theorem}
\begin{remark}
  More explicitly, the theorem asserts that the sets $\{\by\in 
\D(\R_+,\R^\Theta):\,I_{  y}^\ast(\by)\le \beta\}$ are compact for
  all $\beta\ge0$ and that, for any  Borel set $W\subset\D(\R_+,\R^\Theta)$
  such that $\inf_{\by\in
  \text{int}\, W}I_y^\ast(\by)=\inf_{\by\in
  \text{cl}\, W}I_y^\ast(\by)$\,, 
$(1/n)\ln\mathbf P(Y^{(n)}\in W)\to-\inf_{\by\in
  W}I_y^\ast(\by)$\,, as $n\to\infty$\,, where $\text{int}$ and $\text{cl}$ denote the
interior and closure of a set, respectively.
\end{remark}
\begin{remark}
 The limit of the law of large numbers in \eqref{eq:27} follows 
with
$\dot\by(t)=\nabla_\lambda H(\by(t),\lambda)\big|_{\lambda=0}$\,.
In addition, $I_{y}^\ast(\by')=0$ if and only if $\by'=\by$\,, with 
 $\by(0)=y$\,.
\end{remark}

 A proof outline is provided next.
The process $Y^{(n)}$ is a jump semimartingale. The jumps can be of
several kinds:    exogenous class $k$ arrivals at  nodes with population vector
$\theta$ result in jumps $f_{\theta+e_k}/n-f_{\theta}/n$\,,
departures of class $k$ customers from nodes with population vector
$\theta$ produce jumps $f_{\theta-e_k}/n-f_{\theta}/n$\,, whereas
class $k$ customer migrations  from 
nodes $\theta$ to nodes $\theta'$
 give rise to jumps
$f_{\theta'+e_k}/n-f_{\theta'}/n+
f_{\theta-e_k}/n-f_{\theta}/n$\,.
Let 
\begin{equation*}
  \mu^{(n)}([0,t],\Gamma)=\sum_{0<s\le t}\ind_{\{\Delta Y^{(n)}(s)\in\Gamma\}}
\end{equation*}
represent the measure of jumps of $Y^{(n)}$\,, where 
$\Delta Y^{(n)}(s)=Y^{(n)}(s)-Y^{(n)}(s-)$\,, with $Y^{(n)}(s-)$ denoting
the lefthand limit of $Y^{(n)}$ at $s$ and $\Gamma$ standing
for a Borel subset of $\R^\Theta\setminus\{ 0\}\,$\,.
Then, assuming that $\theta\in\Theta^+_k$\,, $\theta\in\Theta^-_k$
and $\theta'\in\Theta^+_k$ on the lefthand sides below, where relevant,
\begin{align*}
    \mu^{(n)}([0,t],\{\frac{f_{\theta+e_k}}{n}-\frac{f_{\theta}}{n}\})
=\int_0^t\sum_{i=1}^{n}
\ind_{\{X^{(n)}_i(s-)=\theta\}}
\,d  N^{(n)}_{i,k}(s)\,,\\
    \mu^{(n)}([0,t],\{\frac{f_{\theta-e_k}}{n}-\frac{f_{\theta}}{n}\})=
\int_0^t\sum_{i=1}^{n}\ind_{\{X^{(n)}_i(s-)=\theta\}}
\sum_{j=1}^\infty\ind_{\{j\le
  X^{(n)}_{i,k}(s-)\}}d  L^{(n)}_{i,k,j}(s)\\
+\int_0^t\sum_{\theta'\in\Theta\setminus Theta^+_k}\sum_{i=1}^{n}\sum_{i'=1}^n\ind_{\{i'\not=i\}}\ind_{\{X^{(n)}_i(s-)=\theta\}}
\ind_{\{X^{(n)}_{i'}(s-)=\theta'\}}
\sum_{j=1}^\infty\ind_{\{j\le
  X^{(n)}_{i,k}(s-)\}}\ind_{\{\xi_{i,k,j}^{(n)}(s)=i'\}}\,d 
R^{(n)}_{i,k,j}(s)\,,\\
  \mu^{(n)}([0,t],\{\frac{f_{\theta'+e_k}}{n}-\frac{f_{\theta'}}{n}+
\frac{f_{\theta-e_k}}{n}-\frac{f_{\theta}}{n}\})\\=
\int_0^t\sum_{i=1}^{n}\sum_{i'=1}^n\ind_{\{i'\not=i\}}\ind_{\{X^{(n)}_i(s-)=\theta\}}
\ind_{\{X^{(n)}_{i'}(s-)=\theta'\}}
\sum_{j=1}^\infty\ind_{\{j\le
  X^{(n)}_{i,k}(s-)\}}\ind_{\{\xi_{i,k,j}^{(n)}(s)=i'\}}\,d 
R^{(n)}_{i,k,j}(s)\,,
\end{align*}
where the $N^{(n)}_{i,k}$\,, $L^{(n)}_{i,k,j}$ and
$R^{(n)}_{i,k,j}$ are  independent  Poisson processes of respective
rates $\alpha_k$\,, $\delta_k$ and $\gamma_k$\,, which are responsible
for customer arrivals, departures and migrations, respectively,
and the $\xi_{i,k,j}^{(n)}(s)$ are
independent random variables  uniformly distributed
in $\{1,2,\ldots,n\}\setminus\{i\}$\,, which are responsible for
 reroutings from node $i$ and which are  independent of the Poisson processes.
The compensators of these measures of jumps relative to the natural
filtration are as follows,
\begin{subequations}
  \begin{align*}
    \nu^{(n)}([0,t],\{\frac{f_{\theta+e_k}}{n}-\frac{f_{\theta}}{n}\})
=n
\int_0^t Y^{(n)}_\theta(s)
\alpha_k\,ds\,,\\
    \nu^{(n)}([0,t],\{\frac{f_{\theta-e_k}}{n}-\frac{f_{\theta}}{n}\})=
n\int_0^tY^{(n)}_{\theta}(s)
\notag\theta_k\delta_k\,ds
\\+\frac{n^2}{n-1}\,\sum_{\theta'\in\Theta}\int_0^tY^{(n)}_{\theta}(s)Y^{(n)}_{\theta'}(s)
\ind_{\{\theta'\in \Theta\setminus \Theta^+_k\}}
\notag\theta_k\gamma_kds
-\frac{n}{n-1}\,\int_0^tY^{(n)}_\theta(s)
\ind_{\{\theta\in\Theta\setminus \Theta^+_k\}}
\theta_k\,\gamma_kds\,, \\
  \nu^{(n)}([0,t],\{\frac{f_{\theta'+e_k}}{n}-\frac{f_{\theta'}}{n}+
\frac{f_{\theta-e_k}}{n}-\frac{f_{\theta}}{n}\})=\frac{n^2}{n-1}
\int_0^tY^{(n)}_\theta(s)Y^{(n)}_{\theta'}(s)\theta_k\,\gamma_k
\,ds\notag\\-
\frac{n}{n-1}
\int_0^tY^{(n)}_\theta(s)
\ind_{\{\theta'=\theta\}}\theta_k\,\gamma_k\,ds\,.
\end{align*}
\end{subequations}
Therefore,
the stochastic cumulant of
$Y^{(n)}$\,, as defined by
 (4.1.14) on p.293 in Puhalskii \cite{Puh01},
 is 
\begin{multline*}
    G^{(n)}_t(\lambda)=\int_0^t\int_{\R^\Theta}(e^{\lambda\cdot u}-1)\nu^{(n)}(ds,du)=
\sum_{k=1}^K\sum_{\theta\in\Theta^+_k}
(e^{(\lambda_{\theta+e_k}-\lambda_\theta)/n}-1)
n\int_0^tY^{(n)}_\theta(s)
\alpha_k\,ds
\\+\sum_{k=1}^K\sum_{\theta\in\Theta^-_k}(e^{(\lambda_{\theta-e_k}-\lambda_\theta)/n}-1)
\bl(n\int_0^tY^{(n)}_{\theta}(s)
\theta_k\delta_k\,ds
+\frac{n^2}{n-1}\sum_{\theta'\in\Theta}\int_0^tY^{(n)}_{\theta}(s)Y^{(n)}_{\theta'}(s)
\ind_{\{\theta'\in\Theta\setminus \Theta^+_k\}}
\theta_k\gamma_kds\\
-\frac{n}{n-1}\int_0^tY^{(n)}_\theta(s)
\ind_{\{\theta\in\Theta\setminus \Theta^+_k\}}
\theta_k\,\gamma_kds\br)+\sum_{k=1}^K
\sum_{\theta\in\Theta^-_k,\,\theta'\in\Theta^+_k}(e^{(\lambda_{\theta'+e_k}-\lambda_{\theta'}
+\lambda_{\theta-e_k}-\lambda_{\theta})/n}-1)\\
\Bl(\frac{n^2}{n-1}
\int_0^tY^{(n)}_\theta(s)Y^{(n)}_{\theta'}(s)\theta_k\,\gamma_k
\,ds-
\frac{n}{n-1}
\int_0^tY^{(n)}_\theta(s)
\ind_{\{\theta'=\theta\}}\theta_k\,\gamma_k
\,ds\Br)\,.\end{multline*}

 The process $Y^{(n)}$ satisfies
 the hypotheses of  Theorem 5.1.5 on
p.357 in Puhalskii \cite{Puh01}.
 In some more detail, since $Y^{(n)}$
is a continuous--time process,    condition $(\sup \mathcal{E})$ on
p.357 in Puhalskii \cite{Puh01} need be checked with
$
\mathcal{E}^{n}_t(\lambda)=e^{G^{(n)}_t( \lambda)}$\,, see (4.1.15) on
p.293 in Puhalskii \cite{Puh01}. 
Recalling that  $Y^{(n)}_\theta(t)$  takes values
in $[0,1]$ implies that the condition in question holds with
\begin{equation*}
  G_t(\lambda;\by)=\int_0^tH(\by(s),\lambda)\,ds
\,.
\end{equation*}
By Theorem 5.1.5 on p.357 in Puhalskii \cite{Puh01}, the sequence
$Y^{(n)}$ is $\C$--exponentially tight in $\D(\R_+,\R^{\Theta})$ and 
its every LD limit point solves
 maxingale problem $( y,G)$\,.

Let  deviability 
$\Pi_y=(\Pi_y(W)\,,W\subset\D(\R_+,\R^\Theta))$ represent an
LD limit point of  $Y^{(n)}$ (recall that   $y^{(n)}\to y$)\,, i.e.,
$\Pi_y(W)\in[0,1]$\,, $\Pi_y(\emptyset)=0$\,, $\Pi_y(
\D(\R_+,\R^\Theta))=1$\,, $\Pi_y(W)=\sup_{\by\in W}\Pi_y(\by)$\,,
 sets
$\{\by\in\D(\R_+,\R^\Theta):\,\Pi_y(\by)\ge\beta\}$ are compact
for $\beta\in(0,1]$\,, where $\Pi_y(\by)=\Pi_y(\{\by\})$\,, and the
sequence of the distributions of $Y^{(n)}$ obeys a subsequential LDP
with deviation function $-\ln\Pi_y(\by)$\,, see Puhalskii
\cite{Puh01,Puh21}.
 Then, $\Pi_y(\by)=0$ unless $\by\in
\C(\R_+,\R^\Theta)$\,, $\by(0)=y$\,, and
$\exp\bl(\lambda\cdot(\by(t)-y)-G_t(\lambda;\by)\br)$ is a local exponential
maxingale in $\C(\R_+,\R^\Theta)$\,, as defined in Puhalskii \cite{Puh01}, where
 $\C(\R_+,\R^\Theta)$ denotes the subset
of continuous functions of
$\D(\R_+,\R^\Theta)$ 
 with the subspace topology which
is the topology of locally uniform convergence. 

 By Lemma 2.7.11 on p.174 in Puhalskii \cite{Puh01},
\begin{equation}
  \label{eq:45a}
  \Pi_y(\by)\le\Pi_y^\ast(\by)\,,
\end{equation}
where 
\begin{equation}
  \label{eq:46a}
  \Pi^\ast_y(\by)=
e^{-I^\ast_y(\by)}\,.
\end{equation}
It is being proved that, in fact,  in \eqref{eq:45a} equality holds.
By \eqref{eq:45a} and \eqref{eq:46a}, it may be assumed that 
$\Pi^\ast_y(\by)>0$\,.
It is  immediate 
that $\Pi_y^\ast(\by)=0$ unless $\by\in\mathbb{ S}_{\abs{\Theta}}$
 so that when proving the equality it may and will be assumed
that $\by(t)=(\by_\theta(t)\,,\theta\in\Theta)\in{\mathbb S}_{\abs{\Theta}}$\,, for all $t$\,.
The equality in \eqref{eq:45a}   is proved, at first, for the case
where $\by$ stays away from the boundary of $\mathbb
S_{\abs{\Theta}}$ so that
$\by_\theta(t)>0$\,, for all $\theta$ and $t$\,, see 
 Lemma \ref{le:identify 1} below.  
Furthermore, if
$\by_\theta(s)>0$ for all $s\in[0,t]$ and one 
 defines, in analogy with  pp.210, 212 in Puhalskii \cite{Puh01}, 
\begin{equation*}
      I^\ast_{y,t}(\by)=\int_0^t
L(\by(s),\dot\by(s))\,ds\,,
 \end{equation*}
provided $\by$ is absolutely continuous, $\by(0)= y$ and 
$\by(s)\in \mathbb S_{\abs{\Theta}}$\,, 
and $I_{y,t}^\ast(\by)=\infty$\,, otherwise, and lets
$\Pi_{y,t}^\ast(\by)=e^{-I^\ast_{y,t}(\by)}$\,, then
$\Pi_y(p_t^{-1}(p_t\by))=\Pi_{y,t}^\ast(\by)$\,, where
$p_t\by=(\by(s\wedge t)\,, s\ge0)$\,, with $u\wedge v=\min(u,v)$\,. 
In order to tackle the case of  trajectories $\by$ that reach the boundary
of the state space, one needs to find trajectories
$\by^\epsilon$ that are locally
 bounded away from zero entrywise and converge to $\by$
locally uniformly, 
 as $\epsilon\to0$\,, such that
\begin{equation}
  \label{eq:54}
\int_0^t L(\by^\epsilon(s),\dot\by^\epsilon(s))\,ds
\to\int_0^t L(\by(s),\dot\by(s))\,ds\,.
\end{equation}
The hard part in the proof of \eqref{eq:54}
is verifying  the hypotheses of  Lebesgue's dominated
convergence theorem.
The needed majoration for the $L(\by^\epsilon(s),\dot\by^\epsilon(s))$
is obtained through the use of a nontrivial bound on the optimisers in  \eqref{eq:11}, see Lemma \ref{le:bound}.
 Upper semicontinuity of $\Pi_y(\by)$ in
 $(y,\by)$ is also important and novel, see Lemma \ref{le:usc}.
 Nevertheless, even then the convergence in \eqref{eq:54} is proved
 for values of $t$ that are not too great. In
order to finish the proof of Theorem \ref{the:LDP}, an arbitrary
trajectory is cut into pieces, for each of which 
$\Pi_y(p_t^{-1}(p_t\by))=\Pi_{y,t}^\ast(\by)$\,,
and
 a Markov property of $\Pi_y$ is used in order to obtain
the needed equality $\Pi_y(\by))=\Pi_{y}^\ast(\by)$\,.

Next,  the 
groundwork is laid by establishing some properties of the function
$L(y,z)$\,.  
Given $\theta,\theta'\in\Theta$\,, 
  a sequence $\theta_0,\ldots,\theta_\ell$ of elements of $\Theta$ is called  a path from
 $\theta$ to $\theta'$ provided $\theta_0=\theta$\,,
 $\theta_\ell=\theta'$ and either 
$\theta_{i+1}=\theta_i+ e_{k_i}$
or $\theta_{i+1}=\theta_i- e_{k_i}$\,, for some
 $k_i$\,, for all  $i\in\{0,1,2,\ldots,\ell-1\}$\,. For $y\in \mathbb S_{\abs{\Theta}}$\,, 
it is said that points $\theta$ and $\theta'$  $y$--communicate if
$y_\theta>0$\,, $y_{\theta'}>0$ and  there
exists a path from $\theta$ to $\theta'$ such that
$y_{\tilde\theta}>0$\,, for every $\tilde\theta$ on the path. The
communication relation is an equivalence relation. The
equivalence classes are denoted by $\Theta_1(y)\,,\ldots,\Theta_{m(y)}(y)$\,.
\begin{lemma}
\label{le:equiv}Let
$y\in\mathbb S_{\abs{\Theta}}$ and  $z\in \R^\Theta$\,.
Suppose that   $z_\theta=0$ when $y_\theta=0$\,.
Then $L(y,z)<\infty$ if and only if $\sum_{\theta\in
  \Theta_i(y)}z_\theta=0$\,, for each $i=1,2,\ldots,m(y)$\,.
If, furthermore,  $y_\theta>0$\,, for all $\theta$\,, then
  supremum in
\eqref{eq:11} is attained.
 \end{lemma}
\begin{proof}
It is proved first  that  $L(y,z)=\infty$ provided
$\sum_{z\in \Theta_1(y)}z_\theta>0$\,. Let
$\lambda_\theta=\Lambda$\,, for $\theta\in \Theta_1(y)$\,, and
$\lambda_\theta=0$\,, for $\theta\notin \Theta_1(y)$\,, where
$\Lambda\to\infty$\,.  Note that
$\sum_{\theta\in\Theta}\lambda_\theta
z_\theta=\Lambda\sum_{\theta\in\Theta_1(y)}z_\theta\to\infty$\,.
Let
\begin{subequations}
  \begin{align}
  \label{eq:5}
  u_k(y,\lambda)=\sum_{\theta\in \Theta^+_k}
e^{\lambda_{\theta+e_k}-\lambda_{\theta}}
 y_\theta\intertext{and}
v_k(y,\lambda)=\sum_{\theta\in \Theta^-_k}
e^{\lambda_{\theta-e_k}-\lambda_{\theta}}
\theta_k y_{\theta}\,.
\label{eq:5a}\end{align}
\end{subequations}
Let also
\begin{equation}
  \label{eq:8}
 \tilde  H(y,\lambda)=
\sum_{k=1}^K\Bl(\alpha_ku_k(y,\lambda)+(\delta_k+
\gamma_k\sum_{\theta'\not\in\Theta^+_k}
y_{\theta'})v_k(y,\lambda)\Br)\,.
\end{equation}
By \eqref{eq:13}, 
\begin{equation*}
    H(y,\lambda)\le
\tilde H(y,\lambda)+ \sum_{k=1}^K
\gamma_k u_k(y,\lambda)v_k(y,\lambda)\,.
\end{equation*}
Noting that if either $\theta+ e_k\in\Theta_1(y)$
or $\theta- e_k\in\Theta_1(y)$\,, whereas
$\theta\not\in\Theta_1(y)$\,, 
then $y_\theta=0$ yields
\begin{align}\label{eq:uk}
    u_k(y,\lambda)=\sum_{\theta\in \Theta_1(y):\,\theta+e_k\not\in\Theta_1(y)}
e^{-\Lambda}
 y_\theta+\sum_{\theta\in \Theta_1(y):\,\theta+e_k\in\Theta_1(y)}
 y_\theta
+\sum_{\theta\not\in\Theta_1(y):\,\theta+e_k\not\in\Theta_1(y)}
 y_\theta\intertext{and}
v_k(y,\lambda)=\sum_{\theta\in \Theta_1(y):\,\theta-e_k\not\in\Theta_1(y)}
e^{-\Lambda}
\theta_ky_{\theta}
+\sum_{\theta\in \Theta_1(y):\,\theta-e_k\in\Theta_1(y)}
\theta_ky_{\theta}+\sum_{\theta\not\in
  \Theta_1(y):\,\theta-e_k\not\in\Theta_1(y)}
\theta_ky_{\theta}\,.\notag
\end{align}
It follows that both $u_k(y,\lambda)$ and $v_k(y,\lambda)$ are bounded
  as $\Lambda\to\infty$\,, so, $H(y,\lambda)$ is bounded. Therefore, $
  \sum_{\theta\in\Theta}\lambda_\theta z_\theta-H(y,\lambda)\to\infty\,.
$

For a sufficiency proof,
suppose that $\sum_{\theta\in\Theta}\lambda_\theta
z_\theta-H(y,\lambda)\to\infty$\,, for some sequence of $\lambda$\,.
Since, by \eqref{eq:13}, \eqref{eq:5}, \eqref{eq:5a} and \eqref{eq:8},
 on recalling that $y\in\mathbb S_\Theta$\,,
\begin{equation*}
      H(y,\lambda)
\ge\tilde H(y,\lambda)-
\sum_{k=1}^K\bl(\alpha_k +
\abs{\Theta}(\delta_k+\gamma_k)
 \br)\,,
\end{equation*}
it follows that
 $\sum_{\theta\in\Theta}\lambda_\theta z_\theta-\tilde H(y,\lambda)\to\infty$\,.
Analogously to \eqref{eq:uk},
\begin{equation*}
      u_k(y,\lambda)\ge\sum_{i=1}^{m(y)}\sum_{\theta\in \Theta:\,\theta+e_k\in\Theta_i(y)}
e^{\lambda_{\theta+e_k}-\lambda_{\theta}}y_{\theta}
=\sum_{i=1}^m\sum_{\theta\in \Theta_i(y):\,\theta+e_k\in\Theta_i(y)}
e^{\lambda_{\theta+e_k}-\lambda_{\theta}}y_{\theta}\,.
\end{equation*}
Applying a similar line of reasoning to $v_k(y,\lambda)$ and
introducing
\begin{subequations}
  \begin{align}
  \label{eq:17}
  u_{i,k}(y,\lambda)=\sum_{\theta\in \Theta_i(y):\,\theta+e_k\in\Theta_i(y)}
e^{\lambda_{\theta+e_k}-\lambda_{\theta}}y_{\theta}
\,, \\
\label{eq:17b}  v_{i,k}(y,\lambda)=\sum_{\theta\in \Theta_i(y):\,\theta-e_k\in\Theta_i(y)}
e^{\lambda_{\theta-e_k}-\lambda_{\theta}}\theta_ky_{\theta}\intertext{ and }
  \tilde H_i(y,\lambda)=
\sum_{k=1}^K\Bl(\alpha_ku_{i,k}(y,\lambda)+(\delta_k+
\gamma_k\sum_{\theta'\not\in\Theta^+_k}
y_{\theta'})v_{i,k}(y,\lambda)\Br)\,,
\label{eq:17a}\end{align}
\end{subequations}
as well as recalling that $z_\theta=0$ when $y_\theta=0$\,, the latter
condition being equivalent to $\theta\notin\cup_{i=1}^m\Theta_i(y)$\,,
  obtains that
\begin{equation*}
    \sum_{\theta\in\Theta}\lambda_\theta z_\theta-\tilde H(y,\lambda)\le
\sum_{i=1}^{m(y)}(\sum_{\theta\in\Theta_i(y)}\lambda_\theta z_\theta-\tilde H_i(y,\lambda))\,.
\end{equation*}
Hence, $\sum_{\theta\in\Theta_i(y)}\lambda_\theta
z_\theta-\tilde H_i(y,\lambda)\to\infty$\,, for some $i$\,. It  is
noteworthy that
 $\tilde H_i(y,\lambda)$ depends on $\lambda$ through
$\lambda_\theta\text{ with }\theta\in\Theta_i(y)$ only\,.

Let $\tilde\theta_i$ represent an element of
$\Theta_i(y)$ with the minimal value of $\lambda_{\theta}$ over
$\theta\in\Theta_i(y)$\,. By passing to subsequences, it  may be assumed that
$\tilde\theta_i$ does not depend on $\lambda$\,.
By \eqref{eq:17}, \eqref{eq:17b}, \eqref{eq:17a} and the fact that
$\sum_{\theta\in \Theta_i(y)} z_\theta=0$, it may be (and will be)
assume that $\lambda_{\tilde\theta_i}=0$ so that $\lambda_\theta\ge0$\,,
for all $\theta\in \Theta_i(y)$\,.
There exists $\hat\theta_i\in\Theta_i(y)$\,, which may be assumed not to
depend on $\lambda$ either, such that  $z_{\hat\theta_i}>0$\,,
$\lambda_{\hat\theta_i}\to\infty$ and
$\lambda_{\hat\theta_i}z_{\hat\theta_i}-
\tilde H_i(y,\lambda)\to\infty$\,.
Let $\theta_0=\hat\theta_i,\theta_1,\ldots,\theta_\ell=\tilde\theta_i$ be
a path
in $\Theta_i(y)$ that
connects $\hat\theta_i$ and $\tilde\theta_i$\,. Since
$\lambda_{\hat\theta_i}=
\sum_{j=1}^\ell(\lambda_{\theta_{j-1}}-\lambda_{\theta_{j}})$ and
 $\lambda_{\hat\theta_i}\to\infty$\,, 
 there exists $j_0$ such that
$\lambda_{\theta_{j_0-1}}-\lambda_{\theta_{j_0}}$ tends to infinity
no slower than $\lambda_{\hat\theta_i}$\,. By \eqref{eq:17},
\eqref{eq:17b}
 and \eqref{eq:17a}, 
\begin{equation*}
     \tilde H_i(y,\lambda)\ge
e^{\lambda_{\theta_{j_0-1}}-\lambda_{\theta_{j_0}}}y_{\theta_{j_0}}
\min_{k}(
\alpha_k \wedge
\delta_k)\,.
\end{equation*}
Since $y_{\theta_{j_0}}>0$\,,
 $\lambda_{\hat\theta_i}
z_{\hat\theta_i}-\tilde H_i(y,\lambda)\to-\infty$\,. 
The contradiction proves the claim.

Suppose now that $y_\theta>0$\,, for all $\theta\in\Theta$\,.
Since $\sum_{\theta\in\Theta}z_\theta=0$\,, it may be assumed that
$\lambda_0=0$\,. If
$\abs{\lambda_{\theta}}\to\infty$\,, for certain $\theta$\,, and
$\abs{\lambda_{\hat\theta}}$ grows the fastest, then
there exist $\tilde\theta$ and $k$ such that 
either $\abs{\lambda_{\tilde\theta+e_k}-\lambda_{\tilde\theta}}$
or $\abs{\lambda_{\tilde\theta-e_k}-\lambda_{\tilde\theta}}$ tends to
infinity at the same rate or faster, which implies that either
$\abs{\lambda_{\hat\theta}z_{\hat\theta}}-
e^{\lambda_{\tilde\theta+e_k}-\lambda_{\tilde\theta}}\alpha_ky_{\tilde\theta}
-e^{\lambda_{\tilde\theta}-\lambda_{\tilde\theta+e_k}}\delta_k
y_{\tilde\theta+e_k}\to-\infty$
or $\abs{\lambda_{\hat\theta}z_{\hat\theta}}-
e^{\lambda_{\tilde\theta-e_k}-\lambda_{\tilde\theta}}\delta_ky_{\tilde\theta}-
e^{\lambda_{\tilde\theta}-\lambda_{\tilde\theta-e_k}}\alpha_k
y_{\tilde\theta-e_k}\to-\infty$\,.
Hence, supremum in \eqref{eq:11} may be taken over a bounded set, so,
it is attained.
   \end{proof}
\begin{lemma}
  \label{le:bound}
There exist $C_1$ and $C_2$ such that
if   $\lambda$ delivers  supremum in \eqref{eq:11}, then
 for  all $k$\,, $\theta\in\Theta^+_k$\,, and $\theta'\in\Theta^-_k$\,,
\begin{equation*}
    e^{\lambda_{\theta+ e_k}-\lambda_\theta}y_\theta
+e^{\lambda_{\theta'- e_k}-\lambda_{\theta'}}y_{\theta'}
+e^{\lambda_{\theta+e_k}-\lambda_{\theta}+\lambda_{\theta'-e_k}-
\lambda_{\theta'}}
  y_{\theta}y_{\theta'}\le C_1+
C_2\sum_{\theta''\in\Theta}\abs{z_{\theta''}}\,.
\end{equation*}
\end{lemma}
\begin{proof}
For $\hat\lambda=(\hat\lambda_{\theta,k})\in\prod_{k\in\{1,2,\ldots,K\}}
\R^{\Theta^+_k}$\,, define
  \begin{multline*}
           \hat H(y,\hat\lambda)=\sum_{k=1}^K\sum_{\theta\in \Theta^+_k}
(e^{\hat\lambda_{\theta,k}}-1)
\alpha_k y_\theta+\sum_{k=1}^K\sum_{\theta\in \Theta^-_k}
(e^{-\hat\lambda_{\theta-e_k,k}}-1)
(\delta_k+
\gamma_k\sum_{\theta'\in\Theta\setminus\Theta^+_k}
y_{\theta'})\theta_ky_{\theta}+
\\+\sum_{k=1}^K\sum_{\theta\in\Theta^-_k,\,\theta'\in \Theta^+_k}
(e^{\hat\lambda_{\theta',k}-\hat\lambda_{\theta-e_k,k}
}-1)
\theta_k\gamma_k y_\theta y_{\theta'}
  \end{multline*}
so that
\begin{equation}
  \label{eq:38}
    L(y,z)=\sup_{\lambda=(\lambda_\theta)\in\R^\Theta}\bl(\lambda\cdot
    z-H(y,\lambda)\br)=
    \sup_{\substack{\lambda=(\lambda_\theta)\in\R^\Theta\,,\\
\hat\lambda=(\hat\lambda_{\theta,k})\in\prod_{k\in\{1,2,\ldots,K\}}
\R^{\Theta^+_k} :\\
\lambda_{\theta+e_k}-\lambda_\theta-\hat\lambda_{\theta,k}=0}}
\bl(\lambda\cdot z-\hat H(y,\hat\lambda)\br)\,.
\end{equation}

Define a Lagrange function, with $r_{\theta,k}\in\R$ and $r=(r_{\theta,k})$\,,
\begin{equation*}
  \mathcal{L}(\lambda,\hat\lambda,r,y,z)=
\sum_{\theta\in\Theta}\lambda_\theta z_\theta-\hat H(y,\hat\lambda)+\sum_{\theta\in\Theta^+_k}
 r_{\theta,k}(\lambda_{\theta+e_k}-\lambda_\theta-\hat\lambda_{\theta,k})\,.  
\end{equation*}
The optimality conditions in \eqref{eq:38} that 
$\partial_{\hat\lambda_{\theta,k}}\mathcal{L}(\lambda,\hat\lambda,r,y,z)=0$ and 
$\partial_{\lambda_{\theta}}\mathcal{L}(\lambda,\hat\lambda,r,y,z)=0$\,, see, e.g., 
Theorem 3.2.2 on p.253 in Alekseev et al. \cite{AleTihFom79},  imply
that, for $\theta\in\Theta^+_k$\,, there exist $r_{\theta,k}$ such that
\begin{multline}
  \label{eq:49}
  -e^{\hat\lambda_{\theta,k}}(\alpha_k 
+\gamma_k\sum_{\theta'\in \Theta^+_k}
e^{-\hat\lambda_{\theta',k}}(\theta_k'+1)
 y_{\theta'+e_k}) y_{\theta}\\
+e^{-\hat\lambda_{\theta,k}}(\theta_k+1)\bl(
\delta_k+
\gamma_k\sum_{\theta'\in\Theta\setminus\Theta_k^+}
y_{\theta'}+
\gamma_k\sum_{\theta'\in \Theta^+_k
  }
e^{\hat\lambda_{\theta',k}}
  y_{\theta'}\br)y_{\theta+e_k}=r_{\theta,k}
\end{multline}
and
\begin{subequations}
    \begin{align}
  \label{eq:50}
  z_\theta+r_{\theta-e_k,k}-r_{\theta,k}=0\,,
\text{ for $\theta\in\Theta^+_k$ with $\theta_k\ge1$\,,}\\
z_\theta-r_{\theta,k}=0\,,\text{ for $\theta\in\Theta^+_k$ with $\theta_k=0$\,.}
  \label{eq:50a}\end{align}
\end{subequations}
Summing in \eqref{eq:49}  yields
\begin{equation}
  \label{eq:60}
    -\sum_{\theta\in\Theta^+_k} e^{\hat\lambda_{\theta,k}}\alpha_k y_{\theta}
+\sum_{\theta\in\Theta^+_k} e^{-\hat\lambda_{\theta,k}}
(\delta_k+
\gamma_k\sum_{\theta'\in\Theta\setminus\Theta^+_k}
y_{\theta'})(\theta_k+1)y_{\theta+e_k}=\sum_{\theta\in\Theta} r_{\theta,k}\,.
\end{equation}
Solving for $\sum_{\theta\in\Theta^+_k} e^{\hat\lambda_{\theta,k}}
y_{\theta}$ and substituting in \eqref{eq:49} imply, after some
algebra,
 that, for
$\theta\in\Theta^+_k$\,, 
\begin{multline*}
          \alpha_ke^{\hat\lambda_{\theta,k}}y_{\theta} 
-e^{-\hat\lambda_{\theta,k}}y_{\theta+e_k}(\theta_k+1)
\bl(\delta_k+
\gamma_k\sum_{\theta'\in\Theta\setminus\Theta^+_k}
y_{\theta'}-
\frac{\gamma_k
\sum_{\theta'} r_{\theta',k}}{\alpha_k 
+\gamma_k\sum_{\theta'\in \Theta^+_k}
e^{-\hat\lambda_{\theta',k}}(\theta_k'+1)
 y_{\theta'+e_k}}\br)\\
 =\frac{-\alpha_kr_{\theta,k}}{\alpha_k 
+\gamma_k\sum_{\theta'\in \Theta^+_k}
e^{-\hat\lambda_{\theta',k}}(\theta_k'+1)
 y_{\theta'+e_k}}\,.
\end{multline*}
Therefore, if $\hat\lambda_{\theta,k}>0$\,, then
\begin{multline*}
          \alpha_ke^{\hat\lambda_{\theta,k}}y_{\theta} \le
y_{\theta+e_k}(\theta_k+1)
\bl(\delta_k+
\gamma_k\sum_{\theta'\in\Theta\setminus\Theta^+_k}
y_{\theta'}+
\frac{\gamma_k
\abs{\sum_{\theta'} r_{\theta',k}}}{\alpha_k 
+\gamma_k\sum_{\theta'\in \Theta^+_k}
e^{-\hat\lambda_{\theta',k}}(\theta_k'+1)
 y_{\theta'+e_k}}\br)\\-
 \frac{\alpha_kr_{\theta,k}}{\alpha_k 
+\gamma_k\sum_{\theta'\in \Theta^+_k}
e^{-\hat\lambda_{\theta',k}}(\theta_k'+1)
 y_{\theta'+e_k}}\le
y_{\theta+e_k}(\theta_k+1)
\bl(\delta_k+
\gamma_k+
\frac{\gamma_k
\abs{\sum_{\theta'} r_{\theta',k}}}{\alpha_k }\br)+
 \abs{r_{\theta,k}}\,,
\end{multline*}
which implies that, no matter the sign of $\hat\lambda_{\theta,k}$\,,
\begin{equation}
  \label{eq:65}
  \alpha_ke^{\hat\lambda_{\theta,k}}y_{\theta}\le\alpha_k y_\theta+
y_{\theta+e_k}(\theta_k+1)
\bl(\delta_k+
\gamma_k+
\frac{\gamma_k
\abs{\sum_{\theta'} r_{\theta',k}}}{\alpha_k }\br)+
 \abs{r_{\theta,k}}\,.
\end{equation}
By \eqref{eq:60} and \eqref{eq:65},
\begin{multline}
  \label{eq:66}
      \sum_{\theta\in\Theta} e^{-\hat\lambda_{\theta,k}}
\delta_k(\theta_k+1)y_{\theta+e_k}\le
\sum_{\theta\in\Theta}\bl(\alpha_k y_\theta+
y_{\theta+e_k}(\theta_k+1)
\bl(\delta_k+
\gamma_k+
\frac{\gamma_k
\abs{\sum_{\theta'} r_{\theta',k}}}{\alpha_k }\br)+
 \abs{r_{\theta,k}}\br)\\
+\sum_{\theta\in\Theta} r_{\theta,k}\,.
\end{multline}
Solving \eqref{eq:50} and \eqref{eq:50a} recursively yields
\begin{equation}
  \label{eq:61}
  r_{\theta,k}=\sum_{i=0}^{\theta_k}z_{\theta-ie_k}\,.
\end{equation}
As a consequence of
\eqref{eq:65}, \eqref{eq:66},  and \eqref{eq:61}, for some $C_1'>0$
and $C_2'>0$\,,
\begin{equation}
  \label{eq:67}
  e^{\hat\lambda_{\theta,k}}y_{\theta}+e^{-\hat\lambda_{\theta,k}}
y_{\theta+e_k}\le C'_1+C'_2\sum_{\theta'\in\Theta}\abs{z_{\theta'}}\,.
\end{equation}
By \eqref{eq:49}, 
 in analogy with \eqref{eq:65},
for $\theta\in\Theta^+_k$\,,
\begin{multline*}
e^{\hat\lambda_{\theta,k}}(\alpha_k 
+\gamma_k\sum_{\theta'\in \Theta^+_k}
e^{-\hat\lambda_{\theta',k}}(\theta_k'+1)
 y_{\theta'+e_k}) y_{\theta}\\\le
    (\theta_k+1)\bl(
\delta_k+
\gamma_k\sum_{\theta'\in\Theta\setminus\Theta^+_k}
y_{\theta'}+
\gamma_k\sum_{\theta'\in \Theta^+_k
  }
e^{\hat\lambda_{\theta',k}}
  y_{\theta'}\br)y_{\theta+e_k}+\abs{r_{\theta,k}}\\
+(\alpha_k 
+\gamma_k\sum_{\theta'\in \Theta^+_k}
e^{-\hat\lambda_{\theta',k}}(\theta_k'+1)
 y_{\theta'+e_k}) y_{\theta}\,.
\end{multline*}
By \eqref{eq:61} and \eqref{eq:67}, there exist
$C_1''>0$ and $C_2''>0$ such that,
 for $\theta\in\Theta^+_k$ and $\theta'\in\Theta^-_k$\,,
\begin{equation*}
  e^{\hat\lambda_{\theta,k}-\hat\lambda_{\theta'-e_k,k}}
  y_{\theta}y_{\theta'}\le C_1''+C_2''\sum_{\theta''\in\Theta}\abs{z_{\theta''}}\,,
\end{equation*}
which concludes the proof on recalling that 
$\hat\lambda_{\theta,k}=\lambda_{\theta+e_k}-\lambda_\theta$
and that
$\hat\lambda_{\theta'-e_k,k}=\lambda_{\theta'}-\lambda_{\theta'-e_k}$\,.
\end{proof}
Denote $\Pi_{y,t}(\by)=\Pi_y(p_t^{-1}(p_t\by))$\,.
\begin{lemma}
  \label{le:identify 1}
Let $\by=(\by(t)\,, t\ge0)$ be an absolutely continuous function
taking values in $\mathbb S_{\abs{\Theta}}$ with $\by(0)= y$\,. If
$\by$ is locally bounded away from zero entrywise , then 
$\Pi_{y,t}(\by)=\Pi_{y,t}^\ast(\by)$\,, for all $t$\,, and
$\Pi_y(\by)=\Pi_y^\ast(\by)$\,.
\end{lemma}
\begin{proof}
By  Lemma \ref{le:equiv}, 
 there exists
function $(\lambda(s)\,,s\ge0)=((\lambda_\theta(s),\theta\in\Theta),\,s\ge0)$ such that, a.e.,
\begin{equation*}
    L(\by(s),\dot{\by}(s))=\sum_{\theta\in\Theta}\lambda_\theta(s)
\dot{\by}_\theta(s)-
H(\by(s),\lambda(s))\,,
\end{equation*}
with ''almost everywhere'' here and below being understood with
respect to the Lebesgue measure.
The following equation is satisfied a.e.:
\begin{equation}
  \label{eq:22}
  \dot{\by}(s)=\nabla_\lambda H(\by(s),\lambda(s))\,.
\end{equation}
Calculations, using \eqref{eq:13}, yield
\begin{multline*}
  \dot\by_\theta(s)=
(\alpha_k+\sum_{\theta'\in \Theta^-_k}
\theta_k'\gamma_ke^{\lambda_{\theta'-e_k}(s)-\lambda_{\theta'}(s)}
\by_{\theta'}(s))e^{\lambda_{\theta}(s)-\lambda_{\theta-e_k}(s)}
 \by_{\theta-e_k}(s)\ind_{\{\theta\in\Theta^-_k\}}\\
+\bl(\delta_k+
\gamma_k\sum_{\theta\in\Theta\setminus\Theta^+_k}
\by_{\theta'}(s)+\gamma_k\sum_{\theta'\in \Theta^+_k}
e^{\lambda_{\theta'+e_k}(s)-\lambda_{\theta'}(s)}
\by_{\theta'}(s)
\br)e^{\lambda_\theta(s)-\lambda_{\theta+e_k}(s)}
(\theta_k+1)\by_{\theta+e_k}(s)\ind_{\{\theta\in\Theta^+_k\}}
\\
-\bl((\alpha_k 
+\sum_{\theta'\in \Theta^-_k}
\theta_k'\gamma_ke^{\lambda_{\theta'-e_k}(s)-\lambda_{\theta'}(s)}
 \by_{\theta'}(s))e^{\lambda_{\theta+e_k}(s)-\lambda_\theta(s)}\ind_{\{\theta\in\Theta^+_k\}}\\+
(\delta_k+
\gamma_k\sum_{\theta'\in\Theta\setminus\Theta^+_k}
\by_{\theta'}(s)+
\gamma_k\sum_{\theta'\in \Theta^+_k}
e^{\lambda_{\theta'+e_k}(s)-\lambda_{\theta'}(s)}
  \by_{\theta'}(s))e^{\lambda_{\theta-e_k}(s)-\lambda_{\theta}(s)}
\theta_k\ind_{\{\theta\in\Theta^-_k\}}\br)
\by_\theta(s)\,.
\end{multline*}
Since the $\by_\theta(s)$ are locally bounded away from zero,
 Lemma \ref{le:bound} implies that the exponentials on the latter 
 righthand side are
 locally integrable functions of $s$\,, so,
 the righthand side of
\eqref{eq:22} is a Lipschitz continuous function of $\by(s)$\,.
It follows that
$\by$ is a unique solution of \eqref{eq:22}.
  By Theorem 2.8.14 on p.213 and Lemma 2.8.20 on p.218 in Puhalskii
  \cite{Puh01}, $\Pi_{y,t}(\by)=\Pi_{y,t}^\ast(\by)$ and
$\Pi_y(\by)=\Pi^\ast_y(\by)$\,.
\end{proof}
  \begin{lemma}
   \label{le:postive}
Let $\by$ be an absolutely continuous function with values in 
$\mathbb S_{\abs
  \Theta}$\,.
Let $\theta^\ast$ represent a point of the maximum of $\by_\theta(0)$ so that 
 $\by_{\theta^\ast}(0)=\max_{\theta\in\Theta}\by_{\theta}(0)$\,.  For
$\epsilon\in(0,1/(3\abs{\Theta}^2))$\,,
let 
$\by^\epsilon_\theta(s)=\by_\theta(s)+ \epsilon$
unless $\theta=\theta^\ast$ and let $\by^\epsilon_{\theta^\ast}(s)=1-
\sum_{\theta\not=
  \theta^\ast}
\by^\epsilon_\theta(s)$\,.
Then, 
 for  $t$
such that $\sum_{\theta\in\Theta}\int_0^t\abs{\dot \by_\theta(s)}\,ds\le
1/(3\abs{\Theta})$\,,  $\by^\epsilon(s)\in\mathbb
S_{\abs{\Theta}}$ on $[0,t]$ and
\begin{equation}
  \label{eq:16}
  \lim_{\epsilon\to0}\int_0^tL(\by^\epsilon(s),\dot\by^\epsilon(s))\,ds
=\int_0^tL(\by(s),\dot\by(s))\,ds\,.
\end{equation}\,.
 \end{lemma}
 \begin{proof}
The functions  $\by^\epsilon_{\theta^\ast}(s)$ are bounded away from
zero on $[0,t]$\,, 
 uniformly over  $\epsilon$\,. Indeed, since $\by_{\theta^\ast}(0)\ge 1/\abs{\Theta}$\,,
$\by_{\theta^\ast}(s)\ge2/(3\abs{\Theta})$ on $[0,t]$\,.
It follows that $\by_{\theta^\ast}^\epsilon(s)\ge
2/(3\abs{\Theta})-\epsilon\abs{\Theta}\ge1/(3\abs{\Theta})$ on
$[0,t]$\,.
Evidently, $\dot \by^\epsilon_\theta(s)=
\dot \by_\theta(s)$ a.e. and 
$\by^\epsilon_\theta(s)\to \by_\theta(s)$ uniformly on bounded sets,
for all $\theta$\,, as $\epsilon\to0$\,.
By Lemma \ref{le:equiv}, 
 a.e., the supremum in \eqref{eq:11} with $y=\by^\epsilon(s)$
and $z=\dot\by^\epsilon(s)$ is attained at some $\lambda^\epsilon(s)$\,.
Since $\by_\theta^\epsilon(s)=\by_\theta(s)+\epsilon$\,, for
$\theta\not=\theta^\ast$\,, and $\by_{\theta^\ast}^\epsilon(s)=
\by_{\theta^\ast}(s)-(\abs{\Theta}-1)\epsilon$\,, 
the definition of $H(y,\lambda)$ in \eqref{eq:13} implies that
\begin{multline*}
       H(\by^\epsilon(s),\lambda)\ge
\sum_{k=1}^K\Bl(\sum_{\theta\in \Theta^+_k}
e^{\lambda_{\theta+e_k}-\lambda_{\theta}}
\alpha_k \by_\theta(s)-
\abs{\Theta}\epsilon e^{\lambda_{\theta^\ast+e_k}-\lambda_{\theta^\ast}}
\alpha_k\Br)\\
-\sum_{k=1}^K\sum_{\theta\in \Theta^+_k}
\alpha_k(\by_\theta(s)+\epsilon)+\sum_{k=1}^K\Bl(
\sum_{\theta\in \Theta^-_k}
e^{\lambda_{\theta-e_k}-\lambda_{\theta}}
(\delta_k+
\gamma_k(\sum_{\theta'\in\Theta\setminus\Theta^+_k}
\by_{\theta'}(s)-\epsilon\abs{\Theta}))\theta_k\by_{\theta}(s)\\
-e^{\lambda_{\theta^\ast-e_k}-\lambda_{\theta^\ast}}
(\delta_k+
\gamma_k)\theta^\ast_k\abs{\Theta}\epsilon\Br)
-\sum_{k=1}^K\sum_{\theta\in \Theta^-_k}
(\delta_k+
\gamma_k\sum_{\theta'\in\Theta\setminus\Theta^+_k}
\by_{\theta'}(s)+\gamma_k\epsilon\abs{\Theta})\theta_k(\by_{\theta}(s)+\epsilon)
\\+\sum_{k=1}^K
\gamma_k\bl(\sum_{\theta'\in \Theta^+_k}
e^{\lambda_{\theta'+e_k}-\lambda_{\theta'}}
 \by_{\theta'}(s)-e^{\lambda_{\theta^\ast+e_k}-\lambda_{\theta^\ast}}
\abs{\Theta}\epsilon
 \br)
\bl(\sum_{\theta\in \Theta^-_k}
e^{\lambda_{\theta-e_k}-\lambda_{\theta}}\theta_k \by_\theta(s)-
e^{\lambda_{\theta^\ast-e_k}-\lambda_{\theta^\ast}}\theta^\ast_k\abs{\Theta}\epsilon\br)
\\-\sum_{k=1}^K\sum_{\theta\in\Theta^-_k,\theta'\in \Theta^+_k}
\theta_k\gamma_k (\by_\theta(s)+\epsilon)( \by_{\theta'}(s)+\epsilon)
\ge 
H(\by(s),\lambda)
-\abs{\Theta}\epsilon\, R(\by(s),\lambda)-\epsilon M\,,
\end{multline*}
where
\begin{multline*}
  R(y,\lambda)=  
\sum_{k=1}^K e^{\lambda_{\theta^\ast+e_k}-\lambda_{\theta^\ast}}
\alpha_k+
\sum_{k=1}^K\sum_{\theta\in\Theta^-_k}
e^{\lambda_{\theta-e_k}-\lambda_\theta}\theta_ky_\theta
+\sum_{k=1}^K
e^{\lambda_{\theta^\ast-e_k}-\lambda_{\theta^\ast}}
(\delta_k+
\gamma_k)\theta^\ast_k\\+
\sum_{k=1}^K
\gamma_ke^{\lambda_{\theta^\ast+e_k}-\lambda_{\theta^\ast}}
\sum_{\theta\in \Theta^-_k}
e^{\lambda_{\theta-e_k}-\lambda_{\theta}}\theta_k y_\theta(s)
+\sum_{k=1}^K
\gamma_k\sum_{\theta\in \Theta^+_k}
e^{\lambda_{\theta+e_k}-\lambda_{\theta}}
 y_{\theta}(s)
e^{\lambda_{\theta^\ast-e_k}-\lambda_{\theta^\ast}}\theta^\ast_k
\end{multline*}
and $M>0$  depends neither on $\lambda$ nor on $\by(s)$\,.

Therefore,
\begin{multline*}
 L(\by^\epsilon(s),\dot\by^\epsilon(s))=
\sum_{\theta\in\Theta}
\lambda^\epsilon_\theta(s)\dot\by^\epsilon_\theta(s)
-H(\by^\epsilon(s),\lambda^\epsilon(s))\le
\sum_{\theta\in\Theta}
\lambda^\epsilon_\theta(s)\dot\by_\theta(s)
-H(\by(s),\lambda^\epsilon(s))\\
+ \abs{\Theta}\epsilon\, R(\by(s),\lambda^\epsilon(s))+\epsilon M
\le L(\by(s),\dot\by(s))+ \abs{\Theta}\epsilon\, R(\by(s),\lambda^\epsilon(s))+\epsilon M\,.
\end{multline*}
Since $\by^\epsilon_{\theta^\ast}(s)$ is locally bounded away from
zero on $[0,t]$ uniformly in $\epsilon$\,, Lemma \ref{le:bound} implies that
\begin{equation*}
\limsup_{\epsilon\to0}  \int_0^t
R(\by(s),\lambda^\epsilon(s))\,ds<\infty\,.
\end{equation*}
Thus,\begin{equation}
  \label{eq:25}
   \limsup_{\epsilon\to0}\int_0^tL(\by^\epsilon(s),\dot\by^\epsilon(s))\,ds\le
 \int_0^tL(\by(s),\dot\by(s))
\,ds\,.
\end{equation}
On the other hand,
by \eqref{eq:13} and \eqref{eq:11},  a.e.,
\begin{equation*}
  \liminf_{\epsilon\to0}L(\by^\epsilon(s),\dot\by^\epsilon(s))\ge
L(\by(s),\dot\by(s))\,.
\end{equation*}
When put together with \eqref{eq:25} and Fatou's lemma, this proves \eqref{eq:16}.\end{proof}
The function $\by^\epsilon$ in the above lemma can be used as an
approximation for $\by$ until $\by_{\theta^\ast}(s)$
 hits 0. At that stage, one starts afresh by
choosing different $\theta$ as $\theta^\ast$\,. The piecing together is done
with the use of the Markov property in the following lemma.
\begin{lemma}
  \label{le:markov}
$(\Pi_y\,, y\in \mathbb S_{\abs{\Theta}})$ is an idempotent Markov
family in the sense that, for $\by\in\D(\R_+,\mathbb S_{\abs{\Theta}})$\,,
\begin{equation*}
    \Pi_y(\by)=\Pi_{y,t}(\by)\Pi_{\by_t}(\vartheta_t\by)\,,
\end{equation*}
where 
$\vartheta_t\by=(\by(s+t)\,,s\ge0)$\,.
\end{lemma}
\begin{proof}
  Let $f(\by)$ represent a nonnegative, bounded and continuous
 function on
  $\D(\R_+,\R^\Theta)$\,. By the Markov property of $Y^{(n)}$\,, on writing $\by=(p_t\by,\vartheta_t\by)$\,,
with $E_{y^{(n)}}$ representing expectation when $Y^{(n)}$ starts at $y^{(n)}$
and 
 $(\mathcal{F}^{(n)}(t)\,,t\ge0)$ representing the filtration
associated with $Y^{(n)}$\,, provided
$y^{(n)}\in\mathbb 
S^{(n)}_{\abs{\Theta}}$\,,
   \begin{multline}
    \label{eq:24a}
    E_{y^{(n)}}(f(Y^{(n)}))^{n}=E_{y^{(n)}}E_{y^{(n)}}((f((p_tY^{(n)},\vartheta_tY^{(n)})))^{n}|\mathcal{F}^{(n)}(t))\\=
E_{y^{(n)}}\bl(E_{Y^{(n)}(t)}f( u,Y^{(n)})^{n}\Big|_{u=p_tY^{(n)}}\br)\,.
  \end{multline}
By Theorem \ref{the:LDP}, if $y^{(n)}\to y$ in $\R^\Theta$ and $u_n\to u$ in
 $\D(\R_+,\R^\Theta)$\,, as $n\to\infty$\,, then
$  (E_{y^{(n)}}f(p_tu_n,Y^{(n)})^{n})^{1/n}\to \sup_{\tilde \by} 
f(p_tu,\tilde \by)\Pi_{y}(\tilde \by)\,.
$ With $g_n( \by)=(E_{ \by(t)}f(p_t \by,Y^{(n)})^{n})^{1/n}$\,, if
$ \by^{(n)}\to  \by$\,, then $g_n(  \by^{(n)})\to
\sup_{\tilde \by} f(p_t \by,\tilde \by)\Pi_{ \by(t)}(\tilde \by)$\,.
Therefore, accounting for \eqref{eq:24a},
\begin{multline*}
  \bl(E_{y^{(n)}}f(Y^{(n)})^{n}\br)^{1/n}=\bl( E_{y^{(n)}}(g_n(Y^{(n)}))^{n}\br)^{1/n}\to
\sup_{ \by}\sup_{\tilde \by} f(p_t \by,\tilde \by)
\Pi_{ \by(t)}(\tilde \by)\Pi_y( \by)\\
=
\sup_{ \by}\sup_{\tilde \by} f(p_t \by,\tilde \by)
\Pi_{ \by(t)}(\tilde \by)\Pi_y(p_t^{-1}(p_t \by))
=
\sup_{ \by,\,\by'} f(p_t \by,\vartheta_t \by')
\Pi_{ \by(t)}(\vartheta_t \by')\Pi_y(p_t^{-1}(p_t \by))\\=\sup_{ \by} f(p_t \by,\vartheta_t \by)
\Pi_{ \by(t)}(\vartheta_t \by)\Pi_y(p_t^{-1}(p_t \by))=
\sup_{ \by} f( \by)
\Pi_{ \by(t)}(\vartheta_t \by)\Pi_y(p_t^{-1}(p_t \by))\,.
\end{multline*}

\end{proof}
\begin{lemma}\label{le:usc}
The function $\Pi_y(\by)$ is upper semicontinuous in $(y,\by)$\,.
\end{lemma}
\begin{proof}
Suppose that initial conditions $Y^{(n)}(0)$ are independent of the
random entities driving the processes $Y^{(n)}$ and satisfy an LDP
in $\R^\Theta$ with a continuous deviation function $I^{ Y}$\,.
 Then the distributions of 
the pairs $( Y^{(n)}(0),Y^{n})$ satisfy a subsequential LDP with
$I^{ Y}(y)-\ln\Pi_y(\by)$\,. Since the latter quantity is lower
semicontinuous in $(y,\by)$ and $I^{ Y}(y)$ is continuous in $y$\,,
$\Pi_y(\by)$ is upper semicontinuous in $(y,\by)$\,.
\end{proof}
\begin{proof}[Proof of Theorem \ref{the:LDP}]
Since $\Pi_y(\by)=\lim_{t\to\infty}\Pi_{y,t}(\by)$ and
$\Pi_y^\ast(\by)=\lim_{t\to\infty}\Pi^\ast_{y,t}(\by)$\,,
it suffices to prove that
\begin{equation}
  \label{eq:5'}
   \Pi_{y,t}(\by)=\Pi^\ast_{y,t}(\by)\,.
\end{equation}
It is  shown, first, that \eqref{eq:5'} holds for all $t$ such that
$\int_0^t\sum_{\theta\in\Theta}\abs{\dot \by_\theta(s)}\,ds\le1/(3\abs{\Theta})$\,.
By Lemma \ref{le:postive}, 
 there 
exist  $\by^\epsilon$\,, such that
$\by^\epsilon_\theta(s)>0$\,, for all $\theta$\,, on $[0,t]$\,,
$\by^\epsilon(s)\to \by(s)$ on $[0,t]$ and 
$\Pi^\ast_{\by^\epsilon(0),t}(\by^\epsilon)\to \Pi^\ast_{y,t}(\by)$\,, as
$\epsilon\to0$\,.  Since, by Lemma \ref{le:identify 1},
$\Pi_{\by^\epsilon(0)}(p_t^{-1}(p_t\by^\epsilon))=\Pi_{\by^\epsilon(0),t}^\ast(\by^\epsilon)$\,,
by upper semicontinuity,
$\Pi_y(p_t^{-1}(p_t\by))\ge\limsup_{\epsilon\to0}\Pi_{\by^\epsilon(0)}(p_t^{-1}(p_t\by^\epsilon))
 =\limsup_{\epsilon\to0}\Pi_{\by^\epsilon(0),t}^\ast(\by^\epsilon)=\Pi_{y,t}^\ast(\by)$\,.
On the other hand,
$\Pi_y(p_t^{-1}(p_t\by))\le\Pi_{y,t}^\ast(\by)$\,, generally, proving
\eqref{eq:5'}.

Given  arbitrary  $t>0$\,, there exist  $0=t_0< t_1<\ldots< t_m=t$
such that $\int_{t_{i-1}}^{t_i}\sum_{\theta\in\Theta} \abs{\dot
  \by_\theta (s)}\,ds\le 1/(3\abs{\Theta})$\,, for all
$i=1,\ldots,m$\,. 
 Consequently,  by the argument in the preceding paragraph,
$\Pi_{\by(t_{i-1})}(p_{t_i-t_{i-1}}^{-1}(p_{t_i-t_{i-1}}(\vartheta_{t_{i-1}}\by)))=
\Pi^\ast_{\by(t_{i-1}),t_i-t_{i-1}}(\vartheta_{t_{i-1}}\by)$\,.
By Lemma \ref{le:markov},
\begin{equation*}
  \Pi_y(p_t^{-1}(p_t\by))=
\prod_{i=1}^{m}\Pi_{\by(t_{i-1})}(p_{t_i-t_{i-1}}^{-1}(p_{t_i-t_{i-1}}(\vartheta_{t_{i-1}}\by)))=
\prod_{i=1}^{m}\Pi^\ast_{\by(t_{i-1}),t_i-t_{i-1}}(\vartheta_{t_{i-1}}\by)=
\Pi^\ast_{y,t}(\by)\,.
\end{equation*}
\end{proof}

\section{Large deviations of the invariant measure. Metastability}
\label{sec:large-devi-invar}
Being irreducible and having a finite state space, the 
process $Y^{(n)}$ possesses a unique invariant measure on $\mathbb
S^{(n)}_{\abs{\Theta}}$\,, see, e.g., Asmussen \cite{Asm03}, which is denoted
by $\mu^{(n)}$\,.  It is convenient to
extend $\mu^{(n)}$ to the whole of $\mathbb S_{\abs{\Theta}}$ by letting 
$\mu^{(n)}(\mathbb S_{\abs{\Theta}}\setminus\mathbb
S^{(n)}_{\abs{\Theta}})=0$\,.
The results in Puhalskii \cite{Puh21} enable one to obtain large
deviation asymptotics of $\mu^{(n)}$.

 Let, for $t>0$\,, $y\in\mathbb S_{\abs{\Theta}}$ and $y'\in\mathbb
 S_{\abs{\Theta}}$\,, 
 \begin{equation*}
   \Phi_t(y,y')=\inf_{\substack{\by\in
      \C(\R_+,\mathbb \R^\Theta):\\
\,\by(0)=y,\,\by(t)=y'}}I^\ast_y(\by)\,.
 \end{equation*}
  Given 
$\nu(\rho)=(\nu_\theta(\rho)\,,\theta\in\Theta)$\,, as 
defined in \eqref{eq:48}, and  $y\in\mathbb S_{\abs{\Theta}}$\,,
let
\begin{equation*}
  \Phi(\nu(\rho),y)=\lim_{t\to\infty}\Phi_t(\nu(\rho),y)=
\inf_{\substack{\by\in
      \C(\R_+,\mathbb \R^\Theta):\\
\,\by(0)=\nu(\rho),\,\by(t)=y\text{ for some }t}}I^\ast_{\nu(\rho)}(\by)\,.
\end{equation*}
(The limit exists because the infima monotonically decrease with
$t$\,, as sitting at $\nu(\rho)$ ``costs'' nothing.) Let $A$ denote the set of solutions $\rho$ of \eqref{eq:48} and
\eqref{eq:53}.
 For  $\rho\in A $\,, let $G(\rho)$ denote the set  of directed graphs
 that are in-trees with root $\rho$
on the vertex set $A$\,. Thus, for every
$\rho'\in A$ and $q\in G(\rho)$\,, 
there is a unique directed path from $\rho'$ to
$\rho$ in $q$\,. For $q\in
G(\rho)$\,,  let $E(q)$ denote the set of edges of $q$\,.
Define 
\begin{equation}
  \label{eq:2}
          J(\nu(\rho))=\inf_{q\in
      G(\rho)}\sum_{(\rho',\rho'')\in E(q)}\Phi(\nu(\rho'),
\nu(\rho''))-\inf_{\tilde \rho\in A }
\inf_{q\in G(\tilde \rho)}\sum_{(\rho',\rho'')\in E(q)}\Phi(\nu(\rho'),\nu(\rho''))\,.
\end{equation}
Let the simplex $\mathbb S_\Theta$ be endowed with the subspace topology.
\begin{theorem}
  \label{inv-LDP}
Suppose that 
  the equations in \eqref{eq:48} and \eqref{eq:53} admit finitely many
solutions $\rho=(\rho_1,\ldots,\rho_K)$\,.
Then, the measures $\mu^{(n)}$ satisfy an LDP in $\mathbb S_{\abs{\Theta}}$ for the topology of
weak convergence
 with a continuous deviation function
\begin{equation}
  \label{eq:36}
  J(y)=\inf_{\rho\in A}(J(\nu(\rho))+\Phi(\nu(\rho),y))\,.
\end{equation}
\end{theorem}
\begin{proof}
  The proof is done by applying Theorem 2.1 in Puhalskii \cite{Puh21}.
Since the set $\mathbb S_{\abs{\Theta}}$ is compact so that the
measures $\mu^{(n)}$ are exponentially tight and $I^\ast(\by)=\infty$
unless $\by\in\C(\R_+,\mathbb S_{\abs{\Theta}})$\,, one needs 
to check the following
  requirements:
  \begin{enumerate}
  \item if $y^{(n)}\to y$\,, then the distributions of $Y^{(n)}$ satisfy
    an LDP with $I^\ast_y$\,,
\item the function $I^\ast_y(\by)$ is lower semicontinuous in
  $(y,\by)$ and the set $\cup_{y\in\mathbb S_{\abs{\Theta}} }
\{\by\in\C(\R_+,\mathbb S_{\abs{\Theta}}):\,I^\ast_y(\by)\le\eta\}$
  is compact\,, for all $\eta\ge0$\,,
\item for all $\by\in\C(\R_+,\mathbb S_{\abs{\Theta}})$\,,
  \begin{equation*}
    I^\ast_y(\by)=\inf_{\by'\in p_s^{-1}(p_s \by)}I^\ast_y(\by')
+I^\ast_{\by(s)}(\vartheta_s\by)\,,
  \end{equation*}
\item\begin{enumerate}
\item if $I^\ast_y(\by)=0$\,, then $\inf_{t\ge0}d(\by(t),A)=0$\,,
where $d$ is a metric on $\mathbb S_{\abs{\Theta}}$\,,
\item if $\by(t)=\nu(\rho)$\,, for all $t\ge0$\,, then
  $I^\ast_{\nu(\rho)}(\by)=0$\,, where $\rho\in A$\,,
\item for any $\rho,\tilde\rho\in A$\,, there exists $t>0$ such that
  $\Phi_t(\nu(\rho),\nu(\tilde\rho))<\infty$\,,
\item for any $\epsilon>0$\,, there exists $\eta>0$ such that if
  $d(y,A)<\eta$\,, then 
$\Phi_{s_0}(y,\nu(\rho))<\epsilon$ and 
$\Phi_{s_1}(\nu(\rho),y)<\epsilon$\,, for some $s_0>0$\,, $s_1>0$ and
$\rho\in A$\,,
\item for any $y\in\mathbb S_{\abs{\Theta}}$ and 
$\epsilon>0$\,, there exist $\eta>0$\,, $t_0$ and $t_1$ such that 
   $\Phi_{t_0}(y,\tilde y)<\epsilon$ and 
$\Phi_{t_1}(\tilde y,y)<\epsilon$ provided  $d(y,\tilde y)<\eta$\,.
\end{enumerate}
  \end{enumerate}
Part 1 holds by Theorem \ref{the:LDP}.
 Part 2 is a consequence of Young's product inequality:
by \eqref{eq:11}\,, for  $s\le t$\,,  $\lambda\in\R^\Theta$\,, 
and $\epsilon>0$\,,
\begin{equation*}
  \lambda\cdot (\by(t)-\by(s))\le\epsilon \int_s^t L(\by(u),\dot\by(u))\,du
+\epsilon\int_s^tH(\by(u),\frac{\lambda}{\epsilon})\,du\,.
\end{equation*}
As $t-s\to0$\,, with $s$ and $t$
being bounded,
the second term on the righthand side goes to $0$ uniformly over $\by$
and over $\lambda$ from a bounded set.
The first term is 
bounded above by $\epsilon\int_0^\infty
L(\by(u),\dot\by(u))\,du\le\epsilon\eta$\,, so, it can be made small
uniformly over $s$ and $t$\,. The needed property holds by
  Arzela--Ascoli's theorem. For part 3,  note that, by \eqref{eq:15},
$\inf_{\by'\in p_s^{-1}(p_s \by)}I^\ast_y(\by')=\int_0^s
L(\by(t),\dot\by(t))\,dt$ and
$I^\ast_{\by(s)}(\vartheta_s\by)=
\int_0^\infty
L(\vartheta_s\by(t),d/dt\,(\vartheta_s\by(t)))\,dt=
\int_s^\infty
L(\by(t),\dot\by(t))\,dt\,.$

As for part 4, Proposition 4 in
Antunes et al. \cite{Rob08} implies that if $\by$ satisfies \eqref{eq:27}, then $\by(t)$
converges, as $t\to\infty$\,, to the set $A$\,,
which verifies the requirement of part 4(a). Part 4(b), essentially,
is about the definition of $\nu(\rho)$\,. For part 4(c), one  can take $t=1$ and
$\by(s)=(1-s\wedge1)\nu(\rho)+s\wedge1\, \nu(\tilde\rho)$\,. Part
4(d) is addressed next. 
Given $\rho$ such that $0<d(y,\nu(\rho))<\eta$\,, one lets
$\by(t)=y+t(\nu(\rho)-y)/d(\nu(\rho),y)$\,. Then, $\by(0)=y$\,, 
$\by(d(\nu(\rho),y))=\nu(\rho)$ and 
\begin{equation}
  \label{eq:42}
\Phi_{d(\nu(\rho),y)}(y,\nu(\rho))\le
\int_0^{d(\nu(\rho),y)}L(\by(t),\dot\by(t))\,dt\,.
\end{equation}
If $t\le d(\nu(\rho),y)$\,, then
 $\by(t)\ge t\nu(\rho)/d(\nu(\rho),y)$ entrywise, so that, on
 recalling that the set of $\rho$ is finite, there exists $\kappa>0$
 such that
$\by_\theta(t)\ge t\kappa/d(\nu(\rho),y) $\,, for  all
$\theta\in\Theta$\,.
By the definition of $H(y,\lambda)$ in \eqref{eq:13},
\begin{multline}
  \label{eq:43}
       H(\by(t),\lambda)\ge
\min_{k}(\alpha_k\wedge\delta_k)\sum_{k=1}^K(\sum_{\theta\in \Theta^+_k}
e^{\lambda_{\theta+ e_k}-\lambda_{\theta}}
+\sum_{\theta\in \Theta^-_k}
e^{\lambda_{\theta- e_k}-\lambda_{\theta}}\theta_k)
\frac{\kappa}{d(\nu(\rho),y)}\,
t\\
-\sum_{k=1}^K
\bl(\alpha_k+
\abs{\Theta}(\delta_k+
\gamma_k) \br)\,.
\end{multline}
Let  $\lambda_{\tilde\theta}=\min_{\theta\in\Theta}\lambda_\theta$ and
$\lambda_{\hat\theta}=\max_{\theta\in\Theta}\lambda_\theta$\,.
Let
$\tilde\theta=\theta_0,\theta_1,\ldots,\theta_\ell=\hat\theta$ represent a
path from $\tilde\theta$ to $\hat\theta$\,.
By Jensen's inequality,
\begin{equation*}
  \sum_{k=1}^K\sum_{\theta\in
  \Theta^{\pm}_k}
e^{\lambda_{\theta\pm e_k}-\lambda_{\theta}}\ge
\sum_{i=1}^\ell
e^{\lambda_{\theta_i}-\lambda_{\theta_{i-1}}}
\ge\ell
e^{\sum_{i=1}^\ell(\lambda_{\theta_i}-\lambda_{\theta_{i-1}})/\ell}
=\ell
e^{(\lambda_{\hat\theta}-\lambda_{\tilde\theta})/\ell}\,.
\end{equation*}
Assuming that $\lambda_{\tilde\theta}<\lambda_{\hat\theta}$ so that $\ell\ge1$
obtains that
\begin{multline*}
  \sum_{\theta\in\Theta}\lambda_\theta\dot{\by}_\theta(t)-H(\by(t),\lambda)
\le\abs{\Theta}( \lambda_{\hat\theta}-\lambda_{\tilde\theta})
\frac{\max_{\theta\in\Theta}\abs{\nu_{\theta}(\rho)-y_\theta}}{d(\nu(\rho),y)}\\-
\min_{k}(\alpha_k\wedge\delta_k)
e^{(\lambda_{\hat\theta}-\lambda_{\tilde\theta})/\abs{\Theta}}
\frac{\kappa}{d(\nu(\rho),y)}\, t+\sum_{k=1}^K
\bl(\alpha_k+
\abs{\Theta}(\delta_k+
\gamma_k) \br)\,.
\end{multline*}
A similar inequality holds if all the $\lambda_\theta$ in \eqref{eq:43}
are the same. Maximisation over
$\lambda_{\hat\theta}-\lambda_{\tilde\theta}$ shows that
 the latter righthand side is bounded above by 
$d_1+d_2\ln(d(\nu(\rho),y)/t)$\,,
for suitable constants $d_1$ and $d_2$\,. Hence,
 the integral on the right of
 \eqref{eq:42} converges to zero as
$\eta\to0$\,. The argument for $\Phi_{s_1}(\nu(\rho),y)$ is
similar: one introduces
$\by(t)=\nu(\rho)+t(y-\nu(\rho))/d(\nu(\rho),y)$\,, notes that
$\by_\theta(t)\ge(1-t)\nu_\theta(\rho)/d(\nu(\rho),y)$  and uses a similar
bound to \eqref{eq:43}.
The checking of part 4(e) is done analogously.
\end{proof}
\begin{remark}
  \label{re:ustkomp}
It is noteworthy that if $\Phi(\nu(\rho'),\nu(\rho''))=0$\,, for some $\rho'$\,,
$\rho''$\,, then $\rho'$ may be omitted in \eqref{eq:36}.
\end{remark}
\begin{remark}
  Interestingly enough,  the quantities $J(\nu(\rho))\,, \rho\in A\,,$ are unique
solutions to the  system of the balance equations that,
for any partition $\{A',A''\}$ of $A$\,,
\begin{equation*}
\inf_{\rho'\in A'}\inf_{\rho''\in A''}
\bl(J(\nu(\rho'))+\Phi(\nu(\rho'),\nu(\rho''))\br)  
=\inf_{\rho'\in A'}\inf_{\rho''\in A''}
\bl(J(\nu(\rho''))+\Phi(\nu(\rho''),\nu(\rho'))\br)  
\end{equation*}
subject to the normalisation condition that 
$\inf_{\rho\in A}J(\nu(\rho))=0$\,, see Puhalskii \cite{Puh21}.
\end{remark}
The next result concerns metastability. It 
is in the spirit of Freidlin and Wentzell \cite{wf2},
see also Shwartz and Weiss \cite{SchWei95}.
 It is also similar to
 Corollary 3.1 in Tibi \cite{Tib10}, where a proof is
outlined assuming a trajectorial LDP. 
As the argument in Tibi \cite{Tib10}  depends
 on certain
contentions  in Freidlin and Wentzell \cite{wf2} being true whose proofs are not
available in the literature, a self--contained proof of
  Theorem \ref{le:exit} is provided  in the appendix.
As before, $P_y$ and 
 $E_y$ denote probability and  expectation, respectively,
 that correspond to the initial
condition
$Y^{(n)}(0)=y$\,. \begin{theorem}
  \label{le:exit}
 Let $\nu(\rho)$ be an  equilibrium of \eqref{eq:27} and let
$D$ be an open   subset of $\mathbb S_\Theta$\,,
  which contains $\nu(\rho)$\,.
Suppose that the solutions of
\eqref{eq:27} with initial conditions in some
 neighbourhood of $ D$ 
converge to $\nu(\rho)$  and stay in $D$ when started in $D$\,. Let 
$\tau^{(n)}=\inf\{t\ge0:\,Y^{(n)}(t)\not\in D\}$\,.
Let $y^{(n)}\in  D\cap S^{(n)}_\Theta$\,.
If  $y^{(n)}\to y\in D$\,, as $n\to\infty$\,, then
\begin{equation*}
  P_{y^{(n)}}\bl(\abs{\frac{1}{n}\,\ln\tau^{(n)}- U}>\kappa\br)\to0
 \end{equation*}
 and
\begin{equation*}
  \frac{1}{n}\,\ln E_{y^{(n)}}(\tau^{(n)})^m\to
mU\,,
\end{equation*}
where $m\in\N$\,, 
\begin{equation*}
   U=\inf_{y'\not\in D}\Phi(\nu(\rho),y')
\end{equation*}
and $\kappa>0$ is otherwise arbitrary.
\end{theorem}
In Antunes et al. \cite{Rob08}   stability of equilibria is
 tackled  via the
Lyapunov function 
\begin{equation*}
  g(y)=\sum_{\theta\in\Theta}y_\theta\ln(\prod_{k=1}^K\theta_k!\,y_\theta)-
\sum_{k=1}^K
\frac{\delta_k+\gamma_k}{\gamma_k}\,
\bl(u\ln u-u)\bigg|^{u=(\alpha_k+\gamma_k\sum_{\theta\in\Theta}\theta_ky_\theta)/(\delta_k+\gamma_k)}_{u=\alpha_k/(\delta_k+\gamma_k)}\,.
\end{equation*}
In the interior of
$\mathbb S_\Theta$\,, see  Antunes et al. \cite{Rob08},
\begin{equation*}
  \nabla g(y)\cdot
V(y)=\sum_{k=1}^K\sum_{\theta\in\Theta}
((\delta_k+\gamma_k)\theta_ky_\theta-b_k(y))
\ln\frac{b_k(y)}{(\delta_k+\gamma_k)\theta_ky_\theta}\,,
\end{equation*}
where
\begin{equation*}
b_k(y)=\alpha_k+\gamma_k\sum_{\theta\in\Theta}\theta_ky_{\theta}\,.
\end{equation*}
Hence, $\nabla g(y)\cdot
V(y)\le0$ so that $g(\by(t))$ is nonincreasing with $t$ along 
solutions of \eqref{eq:27} and $\nabla g(y)\cdot
V(y)<0$ provided $y $ is not an equilibrium  of \eqref{eq:27}.
Furthermore,  $y$ is an equilibrium  of \eqref{eq:27}
if and only if the differential of
$g$\,, as  a function on $\mathbb S_{\abs{\Theta}}$\,, is zero at $y$:
$dg_{_{\mathbb S_{_{\abs{\Theta}}}}}(y)=0$\,.
If $y$ is a local minimum of $g$\,, it is an asymptotically stable equilibrium. 
In order  ``to reduce dimension'', Antunes et al. \cite{Rob08}  introduce the function 
\begin{equation*}
  \phi(\rho)=-\ln
  Z(\rho)+\sum_{k=1}^K\bl(\frac{\gamma_k+\delta_k}{\gamma_k}\,\rho_k-
\frac{\alpha_k}{\gamma_k}\,\ln\rho_k\br)\,.
\end{equation*}
By Theorem 3 in Antunes et al. \cite{Rob08},  $\rho\in\R_+^K$ is 
a local minimum of $\phi$ if and only if $\nu(\rho)$ is a local
minimum of $g$\,; if $\rho$ is 
a saddle point of $\phi$\,, then $\nu(\rho)$ is 
a saddle point of $g$\,.
Besides, a calculation shows that if $\nu(\rho)$ is an equilibrium, then
\begin{equation*}
  g(\nu(\rho))=\phi(\rho)+\sum_{k=1}^K\frac{\alpha_k}{\gamma_k}\bl(\ln\frac{\alpha_k}{\gamma_k+\delta_k}-1\br)\,.
\end{equation*}

An example of bistability along the lines of the one 
in Antunes et al. \cite{Rob08} is analysed next. 
Let $K=2$\,. The polynomial equations for $(\rho_1,\rho_2)$ in
\eqref{eq:48} and
\eqref{eq:53} have finitely many solutions by B\'ezout's theorem as
the polynomials  in the two equations are coprime, see, e.g., Cox, Little
and O'Shea \cite{CoxLittOsh92}.  
Antunes et al. \cite{Rob08} show that, for a
certain choice of parameters
 there are at least two stable equilibria.  Suppose that class 1
 customers require one unit of capacity, so, $A_1=1$ whereas class 2
 customers require the whole capacity, so, $A_2=C$\,. Accordingly, 
class 1 and class 2 customers cannot coexist at the same node.
 It stands to reason that  there could be two 
stable states  where
  class 1 customers are prevalent or   class 2 customers
 are prevalent, respectively. This is substantiated next.

By hypotheses, $\theta_1$ takes values in the set
$\{0,1,\ldots,C\}$ and $\theta_2$ takes values in $\{0,1\}$\,.
Then, with $\rho=(\rho_1,\rho_2)$\,, 
\begin{equation}
  \label{eq:55}
  Z(\rho)=\sum_{i=0}^C \frac{\rho_1^i}{i!}+\rho_2
\end{equation}
and
\begin{equation*}
  \phi(\rho)=-\ln\bl(\sum_{i=0}^C\frac{\rho_1^i}{i!}+\rho_2\br)+
\frac{\gamma_1+\delta_1}{\gamma_1}\,\rho_1+\frac{\gamma_2+\delta_2}{\gamma_2}\,
\rho_2-\frac{\alpha_1}{\gamma_1}\,\ln\rho_1-\frac{\alpha_2}{\gamma_2}\,\ln\rho_2\,.
\end{equation*}
By \eqref{eq:48} and \eqref{eq:53} with $k=1$\,,
\begin{equation}
  \label{eq:1}
  Z(\rho)=\frac{\gamma_1\rho_1(\rho_2+\rho_1^C/C!)}{\alpha_1-
\delta_1\rho_1}
\end{equation}
(as $\sum_{\theta\in\Theta}\theta_1\nu_\theta(\rho)<\rho_1$\,,
$\alpha_1>\delta_1\rho_1$). By \eqref{eq:48} and \eqref{eq:53} with $k=2$\,, $\rho_2$ satisfies the
quadratic equation
\begin{multline*}
  \gamma_1(\gamma_2+\delta_2)\rho_2^2+
(\gamma_1(\gamma_2+\delta_2)\frac{\rho_1^{C}}{C!}-\alpha_2\gamma_1
-\gamma_2(\frac{\alpha_1}{\rho_1}-\delta_1))\rho_2-
\alpha_2\gamma_1\,\frac{\rho_1^{C}}{C!}=0\,.
\end{multline*}

Hence,
\begin{multline}
  \label{eq:18}
  \rho_2(\rho_1)=\frac{1}{2\gamma_1(\gamma_2+\delta_2)}
\Bl(-\gamma_1(\gamma_2+\delta_2)\frac{\rho_1^{C}}{C!}+\alpha_2\gamma_1
+\gamma_2(\frac{\alpha_1}{\rho_1}-\delta_1)\\+
\sqrt{\bl(\gamma_1(\gamma_2+\delta_2)
\frac{\rho_1^{C}}{C!}-\alpha_2\gamma_1
-\gamma_2(\frac{\alpha_1}{\rho_1}-\delta_1)\br)^2
+4\gamma^2_1(\gamma_2+\delta_2)\alpha_2\,\frac{\rho_1^{C}}{C!}}\;\Br)\,.
\end{multline}
Equating  the righthand sides of \eqref{eq:55} and
\eqref{eq:1} 
yields
\begin{equation*}
  h(\rho)=0\,,
\end{equation*}
where
\begin{equation}  \label{eq:59}
  h(\rho)=\sum_{i=0}^C \frac{\rho_1^i}{i!}+\rho_2-
\frac{\gamma_1\rho_1}{\alpha_1-\delta_1\rho_1}
\bl(\frac{\rho_1^{C}}{C!}+\rho_2\br)\,.
\end{equation}
Let $C=20$\,,  $\alpha_1=.5$\,, $\alpha_2=9$\,,
$\gamma_1=\gamma_2=1$ and $\delta_1=\delta_2=.01$\,.
(The parameters for which Antunes et al. \cite{Rob08} show the
existence of two stable equilibria are $C=20$\,, $\alpha_1=.68$\,,
$\alpha_2=9$\,, $\gamma_1=\gamma_2=1$ and $\delta_1=\delta_2=0$\,.
It is of interest to allow nonzero $\delta_1$ and (or)
$\delta_2$\,.) For $\rho_1$ close to zero, the  term $\alpha_1/\rho_1$
on the right of
  \eqref{eq:18} dominates,  $\rho_2$
decreases rapidly as $\rho_1$ increases,
 so does the righthand side of \eqref{eq:59}, its first zero being 
 $\rho_1\approx.5966$\,. 
The righthand side of \eqref{eq:59} keeps decreasing as the leftmost sum  starts
 taking over until it reaches a
 minimum of approximately $-23.556$ at $\rho_1\approx2.8861$
and begins to increase and crosses the zero level for a
 second time for $\rho_1\approx4.1786$\,. It keeps growing reaching a
 maximum of approximately $1.5794\cdot 10^5$ at 
$\rho_1\approx12.896$ until another
 change of a dominating term when 
$-\gamma_1\rho_1/(\alpha_1-\rho_1\delta_1)\,\rho_1^{C}/C!$
 takes over. The righthand side of
\eqref{eq:59} then plunges to $-\infty$\,,  crossing  the zero
level for a third time at  $\rho_1\approx13.72715$ in the process.
Graphs in Fig.1 and Fig.2 provide an illustration.
\begin{figure}[h!]
  \centering
  \includegraphics[scale=.7]{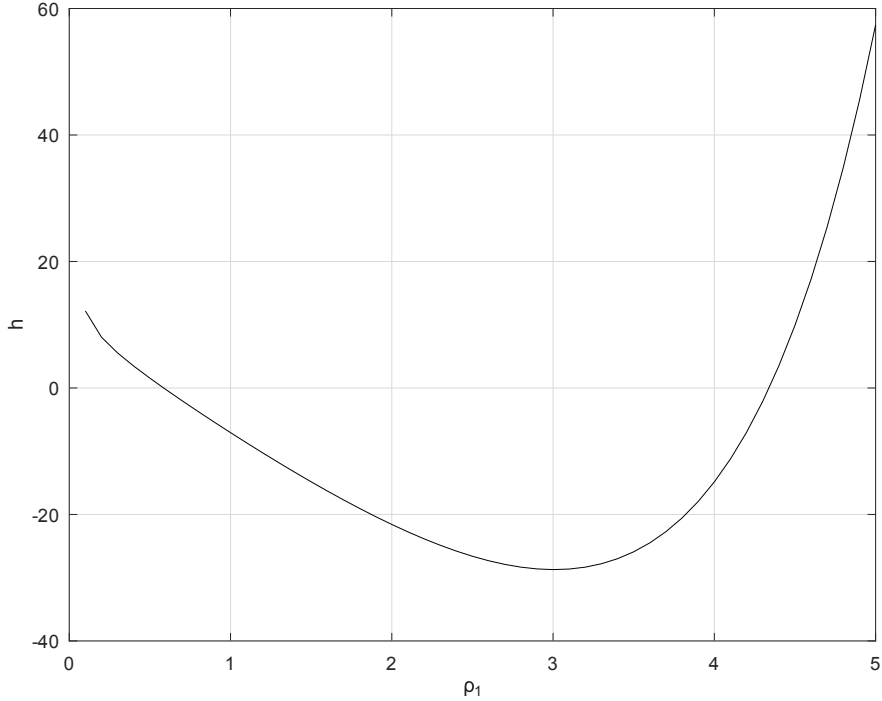}
  \caption{Function $h(\rho_1,\rho_2(\rho_1))$ for $\rho_1$ small\,}
      \label{fig:1}
\end{figure}
\begin{figure}[h!]
  \centering
  \includegraphics[scale=.7]{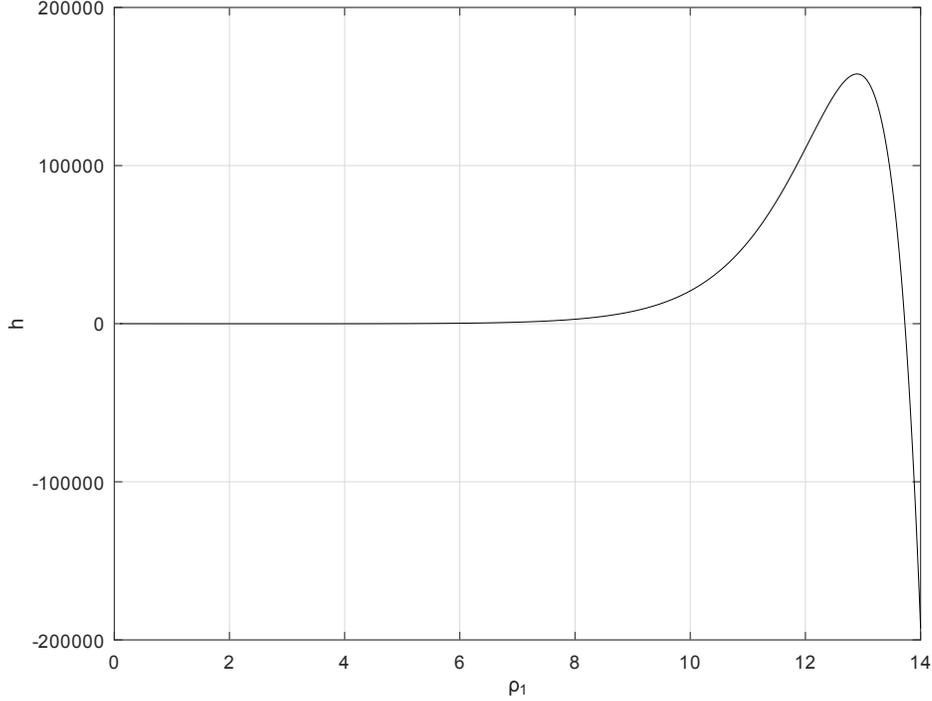}
  \caption{Function $h(\rho_1,\rho_2(\rho_1))$ globally}
      \label{fig:2}
\end{figure}
 Thus, all in all, there are three equilibria:
$\rho^{(1)}\approx(.5966,8.8293)$\,, $\rho^{(2)}\approx
(4.1786,8.1115)$\,, and  $\rho^{(3)}\approx(13.72715,8.9906)$\,.
  The  second derivatives of $\phi$  are
$\nabla^2\phi(\rho^{(1)})\approx\Bl(\begin{array}{cc}
       1.2633&0.016025\\0.016025&0.1243\end{array}\Br)$\,,
 $\nabla^2\phi(\rho^{(3)})\approx\Bl(\begin{array}{cc}
       0.015412&1.1085\cdot 10^{-6}\\1.1085
           \cdot 10^{-6}&0.1113 \end{array}\Br)$\,, and
$ \nabla^2\phi(\rho^{(2)})\approx\Bl(\begin{array}{cc}
       -0.069679&0.0121200\\0.012120&0.1370\end{array}\Br)$\,,
the eigenvalues in the latter case being
  $-0.070387$ and $0.137708$ approximately.
With $\nabla^2\phi(\rho^{(1)})$ and $\nabla^2\phi(\rho^{(3)})$ being
positive definite, $\rho^{(1)}$ and $\rho^{(3)}$ are local minima, whereas 
$\rho^{(2)}$ is a saddle point.
A 3D mesh plot of $\phi(\rho)$ with a contour plot underneath
 is in Fig.\ref{fig:6}.
\begin{figure}[h!]
  \centering
  \includegraphics[scale=.7]{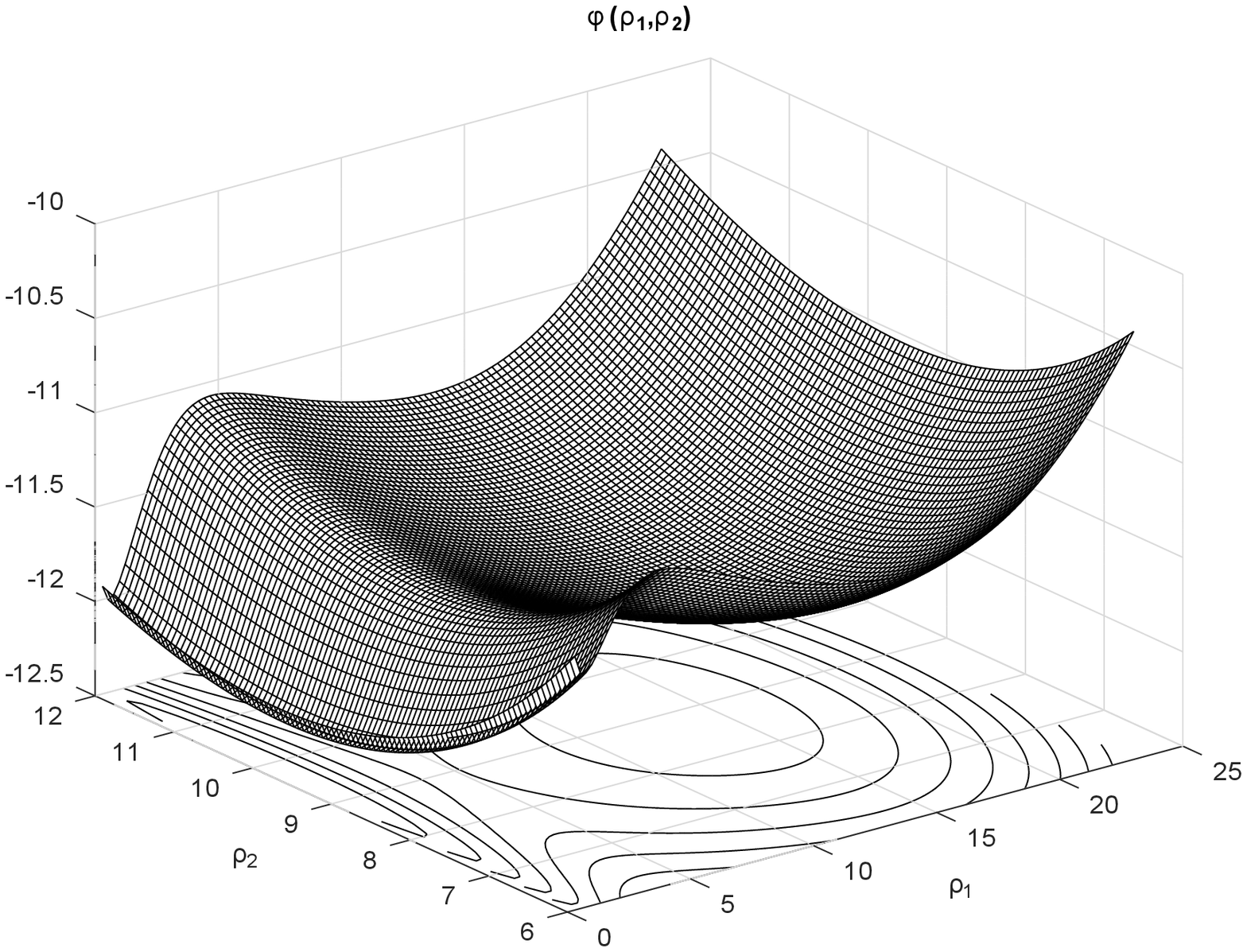}
  \caption{3D mesh plot and  contour plot of $\phi(\rho_1,\rho_2)$\,}
      \label{fig:6}
\end{figure}
Therefore, $\nu(\rho^{(1)}))$ and $\nu(\rho^{(3)})$ are asymptotically
stable
equilibria 
 so that
 the network process spends  exponentially long periods of time
   in the neighbourhoods of those equilibria, while 
$\nu(\rho^{(2)})$ is an unstable equilibrium.

The expected number of class $k$ customers being
\begin{equation*}
  EQ_k=\rho_k(1-\sum_{\theta:\,\sum_{k'} A_{k'}\theta_{k'}>C- A_k}\nu_\theta(\rho))
\end{equation*}
implies that the average numbers of class 1 and class 2 customers are
\begin{equation*}
  EQ_1=\rho_1(1-\frac{1}{Z(\rho)}\frac{\rho_1^C}{C!})
\end{equation*}
and
\begin{equation*}
  EQ_2=\rho_2(1-\frac{1}{Z(\rho)}
\bl(\sum_{\theta_1=1}^C\frac{\rho_1^{\theta_1}}{\theta_1!}+\rho_2\br))=\frac{\rho_2}{Z(\rho)}\,,
\end{equation*}
respectively.
For the stable equilibria, 
calculations yield $(EQ_1^{(1)},EQ_2^{(1)})\approx(.5966,.8281)$
and $(EQ_1^{(3)},EQ_2^{(3)})\approx(13.365,1.0235\cdot 10^{-5})$\,.
Thus, for $\rho=\rho^{(1)}$\,, class 2 customers are prevalent and, for 
$\rho=\rho^{(3)}$\,, class 1 customers are prevalent.
(The pattern of $h(\rho)$
first decreasing, then increasing and
decreasing again is sensitive to the values of $\delta_1$ and
$\delta_2$\,. When $\delta_1=\delta_2=.1$\,, only a downward trend is
present, so, there is only one equilibrium.)

Also, calculations yield 
$\phi(\rho^{(1)})=-12.284$\,, $\phi(\rho^{(2)})=-11.560$\,,
$\phi(\rho^{(3)})=-12.043$\,
so that $g(\nu(\rho^{(2)}))>g(\nu(\rho^{(1)}))\vee g(\nu(\rho^{(3)}))$\,.Thus, if  $\by(0)=\nu(\rho^{(2)})$ experiences a small displacement $\Delta$
at time zero in a
direction collinear with the direction of
 the eigenvector with a negative eigenvalue, then
$g(y+\Delta)<g(y)$\,, so that the associated trajectory $\by(t)$ will
end up in one of the equilibria $\nu(\rho^{(1)})$  or
$\nu(\rho^{(3)})$\,.
Hence, either $\Phi(\nu(\rho^{(2)}),\nu(\rho^{(1)})=0$
or $\Phi(\nu(\rho^{(2)}),\nu(\rho^{(3)}))=0$\,.

On denoting $\nu^{(i)}=\nu(\rho^{(i)})$\,, by \eqref{eq:2},
\begin{align*}
    J(\nu^{(1)})=(\Phi(\nu^{(2)},\nu^{(1)})+\Phi(\nu^{(3)},\nu^{(2)}))
\wedge (\Phi(\nu^{(2)},\nu^{(1)})+\Phi(\nu^{(3)},\nu^{(1)}))\\\wedge
(\Phi(\nu^{(3)},\nu^{(1)})+\Phi(\nu^{(2)},\nu^{(3)}))-\Psi\,,\\
  J(\nu^{(2)})=(\Phi(\nu^{(1)},\nu^{(2)})+\Phi(\nu^{(3)},\nu^{(1)}))
\wedge (\Phi(\nu^{(1)},\nu^{(2)})+\Phi(\nu^{(3)},\nu^{(2)}))\\\wedge
(\Phi(\nu^{(3)},\nu^{(2)})+\Phi(\nu^{(1)},\nu^{(3)}))-\Psi\,,\\
 J(\nu^{(3)})=(\Phi(\nu^{(2)},\nu^{(3)})+\Phi(\nu^{(1)},\nu^{(2)}))
\wedge (\Phi(\nu^{(2)},\nu^{(3)})+\Phi(\nu^{(1)},\nu^{(3)}))\\\wedge
(\Phi(\nu^{(1)},\nu^{(3)})+\Phi(\nu^{(2)},\nu^{(1)}))-\Psi\,,
\end{align*}
where
\begin{align*}
  \Psi=(\Phi(\nu^{(2)},\nu^{(1)})+\Phi(\nu^{(3)},\nu^{(2)}))
\wedge (\Phi(\nu^{(2)},\nu^{(1)})+\Phi(\nu^{(3)},\nu^{(1)}))\wedge
(\Phi(\nu^{(3)},\nu^{(1)})+\Phi(\nu^{(2)},\nu^{(3)}))\\\wedge
(\Phi(\nu^{(1)},\nu^{(2)})+\Phi(\nu^{(3)},\nu^{(1)}))
\wedge (\Phi(\nu^{(1)},\nu^{(2)})+\Phi(\nu^{(3)},\nu^{(2)}))\wedge
(\Phi(\nu^{(3)},\nu^{(2)})+\Phi(\nu^{(1)},\nu^{(3)}))\\\wedge
(\Phi(\nu^{(2)},\nu^{(3)})+\Phi(\nu^{(1)},\nu^{(2)}))
\wedge (\Phi(\nu^{(2)},\nu^{(3)})+\Phi(\nu^{(1)},\nu^{(3)}))\wedge
(\Phi(\nu^{(1)},\nu^{(3)})+\Phi(\nu^{(2)},\nu^{(1)}))\,.
\end{align*}
As a consequence, if $\Phi(\nu^{(2)},\nu^{(1)})=0$\,, then
\begin{align*}
    J(\nu^{(1)})=\Phi(\nu^{(3)},\nu^{(2)})
\wedge \Phi(\nu^{(3)},\nu^{(1)})-\Psi\,,\\
  J(\nu^{(2)})=(\Phi(\nu^{(1)},\nu^{(2)})+\Phi(\nu^{(3)},\nu^{(1)}))
\wedge (\Phi(\nu^{(1)},\nu^{(2)})+\Phi(\nu^{(3)},\nu^{(2)}))\\\wedge
(\Phi(\nu^{(3)},\nu^{(2)})+\Phi(\nu^{(1)},\nu^{(3)}))-\Psi\,,\\
 J(\nu^{(3)})=(\Phi(\nu^{(2)},\nu^{(3)})+\Phi(\nu^{(1)},\nu^{(2)}))
\wedge 
\Phi(\nu^{(1)},\nu^{(3)})-\Psi
\end{align*}
and
\begin{align*}
  \Psi=\Phi(\nu^{(3)},\nu^{(2)})
\wedge \Phi(\nu^{(3)},\nu^{(1)})
\wedge
(\Phi(\nu^{(2)},\nu^{(3)})+\Phi(\nu^{(1)},\nu^{(2)}))
\wedge
\Phi(\nu^{(1)},\nu^{(3)})\,.
\end{align*}
It is being conjectured that, furthermore,
$\Phi(\nu^{(2)},\nu^{(1)})=\Phi(\nu^{(2)},\nu^{(3)})=0$
and that $\Phi(\nu^{(3)},\nu^{(2)})\le
\Phi(\nu^{(3)},\nu^{(1)})$
and $\Phi(\nu^{(1)},\nu^{(2)})\le
\Phi(\nu^{(1)},\nu^{(3)})$\,, in which case the expressions above
simplify. \appendix
\section{Proof of Theorem \ref{le:exit} }
\label{sec:appendix}
The proof is along the lines of the developments in Freidlin and Wentzell
\cite{wf2} and Shwartz and Weiss \cite{SchWei95}.
Suppose, it has been proved that
\begin{align}
  \label{eq:9}
  \limsup_{n\to\infty}\frac{1}{n }\,
\ln E_{y^{(n)}}(\tau^{(n)})^m\le m
U\intertext{ and }
\lim_{n\to\infty}P_{y^{(n)}}\bl(\frac{1}{n}\,\ln\tau^{(n)}\le
  U-3\kappa\br)=0\,.
\label{eq:9'}
\end{align}
By Markov's inequality and \eqref{eq:9} with $m=1$\,,
\begin{equation*}
\limsup_{n\to\infty}  P_{y^{(n)}}\bl(\frac{1}{n}\,\ln\tau^{(n)}\ge
  U+\kappa\br)^{1/n}\le
\limsup_{n\to\infty}( E \tau^{(n)})^{1/n}e^{-(U+\kappa)}\le e^{-\kappa}
\end{equation*}
so that 
\begin{equation*}
    \lim_{n\to\infty}  P_{y^{(n)}}\bl(\frac{1}{n}\,\ln\tau^{(n)}\ge
  U+\kappa\br)=0\,.
\end{equation*}
By Jensen's inequality, for $\epsilon>0$\,,
\begin{equation*}
  \frac{1}{n }\,
\ln E_{y^{(n)}}(\tau^{(n)}\vee\epsilon)^m\ge \frac{m}{n }\,
E_{y^{(n)}}(\ln \tau^{(n)}\vee\ln\epsilon)\,.
\end{equation*}
By \eqref{eq:9'} and Fatou's lemma, 
\begin{equation*}
  \liminf_{n\to\infty}
E_{y^{(n)}}\frac{1}{n }\,(\ln \tau^{(n)}\vee\ln\epsilon)
\ge U\,.
\end{equation*}
On the other hand,
\begin{equation*}
  \ln E_{y^{(n)}}(\tau^{(n)}\vee\epsilon)^m\le 
\ln(E_{y^{(n)}}(\tau^{(n)})^m+\epsilon^m)\le
\ln2 +\ln( E_{y^{(n)}}(\tau^{(n)})^m)\vee\ln\epsilon^m\,.
\end{equation*}
Hence,
\begin{equation*}
  \liminf_{n\to\infty}\frac{1}{n}\,\ln E_{y^{(n)}}(\tau^{(n)})^m\ge m U\,.
\end{equation*}
Next, \eqref{eq:9} and \eqref{eq:9'} are proved.

For $r>0$\,, let $B_r$ denote the
open ball of radius $r$ about $\nu(\rho)$\,. 
Let $\eta>0$ be small enough for $B_{3\eta}$ to belong to $D$\,.
Let $T(y')$ denote the length of time that it takes
the solution $\by$ of \eqref{eq:27} with $y'$ as an initial point to reach
$\cl\! B_{\eta/2}$\,, where $y'\in\cl\! D$\,. The
function $T(y')$ is upper semicontinuous. Since $D$ is bounded, so is
$T(y')$\,. Let $T_1=\max_{y'\in\cl D}T(y')$\,. Let
$\sigma^{(n)}$ represent the first time when
$Y^{(n)}$ reaches  the closed  ball $\cl\!B_\eta$\,.
Let $\by(t)$ solve \eqref{eq:27} with $\by(0)=y$\,.
Since
 $y^{(n)}\to y$ and $Y^n(t)\to \by(t)$ uniformly on $[0,T_1]$ in
$P_{y^{(n)}}$--probability,
it may be assumed that 
\begin{equation}
\label{eq:23}  P_{y^{(n)}}(\sigma^{(n)}\le T_1)\ge \frac{1}{2}\,,
\end{equation}
provided $n$ is great enough.

Lemma 3.1 in Puhalskii \cite{Puh21} implies that $\Phi(\nu(\rho),y')$ is a
continuous function of $y'$\,.
Therefore,
\begin{equation}
  \label{eq:30}
  U=\inf_{y'\not\in \cl D}\Phi(\nu(\rho),y')\,.
\end{equation}
Let $D_\eta$ represent the open $\eta$--neighbourhood of $D$ and let
\begin{equation*}
  U_\eta=\inf_{y\not\in D_\eta}\Phi(\nu(\rho),y)\,.
\end{equation*}
By \eqref{eq:30}, given $\beta>0$\,, 
one may assume that $\eta$ is small enough so that
$U_\eta\le U+\beta$\,. (One can
assume that $U<\infty$\,.) Furthermore, given $y'^{(n)}\in B_\eta$\,,
there exist $\by^{(n)}$ and $t^{(n)}$
 such that $\by^{(n)}(0)=y'^{(n)}$\,, $\by^{(n)}(t^{(n)})\not\in D_\eta$ and
 $I^\ast_{y'^{(n)}}(\by^{(n)})\le U_\eta+\beta$\,.
(Note that one can get from $y'^{(n)}$ to $\nu(\rho)$ at an arbitrarily small
cost by following the solution of \eqref{eq:27}.)  

It is shown next that the sequence $t^{(n)}$ can be chosen bounded.
Firstly, $t^{(n)}$ can be chosen as the smallest $t$ with
$\by^{(n)}(t)\not\in D_\eta$\,.
 Let $s^{(n)}$ be the last time $t$
such that $\by^{(n)}(t)\in \cl\! B_\eta$\,. Then,
for $t\in[s^{(n)},t^{(n)}]$\,, the function
$\by^{(n)}(t)$ takes values in
the closed set $\cl\! D_\eta\setminus B_\eta$\,.  Let $T$ denote 
the maximal length of
time it takes a solution of \eqref{eq:27} with an initial point in
$\cl\! D_\eta\setminus B_\eta$ to get to
$\cl\!B_\eta$\,. If $t^{(n)}- s^{(n)}\le NT$\,, for some $N$\,, the
proof is over. Otherwise,  on denoting, for $W\subset
\C(\R_+,\R^{\abs{\Theta}})$\,,  $I^\ast(W)=\inf_{\by\in W} I^\ast(\by)$\,,
 by \eqref{eq:15},
\begin{multline*}
  I^\ast_{y'^{(n)}}(\by^{(n)})\ge
\inf_{y'\in \cl\!D_\eta\setminus B_\eta} I_{y'}^\ast(\vartheta_{s^{(n)}}\by^{(n)})=
\inf_{y'\in \cl\!D_\eta\setminus B_\eta}
\bl(I^\ast_{y'}( p_T^{-1}(p_T(\vartheta_{s^{(n)}}\by^{(n)}))) +
I^\ast_{\by ^{(n)}(s^{(n)}+T)}(\vartheta_{s^{(n)}+T}\by^{(n)})\br)
\\\ge \inf_{y'\in \cl\!D_\eta\setminus B_\eta}
I^\ast_{y'}(p_T^{-1}(p_T(\vartheta_{s^{(n)}}\by^{(n)}))) +
\inf_{y'\in \cl\!D_\eta\setminus B_\eta}
I^\ast_{y'}(\vartheta_{s^{(n)}+T}\by^{(n)})\\
=\inf_{y'\in \cl\!D_\eta\setminus B_\eta}
I^\ast_{y'}(p_T^{-1}(p_T(\vartheta_{s^{(n)}}\by^{(n)}))) +
\inf_{y'\in \cl\!D_\eta\setminus B_\eta}
(I^\ast_{y'}(p_T^{-1}(p_T\vartheta_{s^{(n)}+T}\by^{(n)}))+
I^\ast_{\by ^{(n)}(s^{(n)}+2T)}(\vartheta_{s^{(n)}+2T}\by^{(n)}))\\\ge
\inf_{y'\in \cl\!D_\eta\setminus B_\eta}
I^\ast_{y'}(p_T^{-1}(p_T(\vartheta_{s^{(n)}}\by^{(n)}))) +
\inf_{y'\in \cl\!D_\eta\setminus B_\eta}
I^\ast_{y'}(p_T^{-1}(p_T\vartheta_{s^{(n)}+T}\by^{(n)}))\\
+\inf_{y'\in \cl\!D_\eta\setminus B_\eta}
I^\ast_{y'}(\vartheta_{s^{(n)}+2T}\by^{(n)})\,.
\end{multline*}
Continuing on yields, for arbitrary $N\in\N$ such that
$t^{(n)}-s^{(n)}\ge NT$\,,
\begin{multline*}
    I^\ast_{y'^{(n)}}(\by^{(n)})\ge\sum_{m=0}^N
\inf_{y'\in\cl\! D_\eta\setminus B_\eta}
I^\ast_{y'}(p_T^{-1}(p_T\vartheta_{s^{(n)}+mT}\by^{(n)}))\ge
N
\inf_{y'\in\cl\! D_\eta\setminus B_\eta}
\inf_{\substack{\by':\,\by'(t)\in\cl\! D_\eta\setminus B_\eta\\
\text{ for all }t\in[0,T]}}
I^\ast_{y'}(\by')\,.
\end{multline*}
The latter infimum is positive because no solution of \eqref{eq:27}
 belongs to the set
$\{\by':\,\by'(t)\in\cl\! D_\eta\setminus B_\eta
\text{ for all }t\in[0,T]\}$\,.
It follows that the values of $N$ have to be bounded, so, the
$t^{(n)}-s^{(n)}$ have to be bounded. 

The sequence $s^{(n)}$ can be assumed bounded. Indeed, one can change $\by^{(n)}$ by
replacing the piece of $\by^{(n)}$ on 
$[0,s^{(n)}]$ with a straight line segment connecting $y'^{(n)}$ and
$\by^{(n)}(s^{(n)})$\,. The modified trajectory
is given by $\tilde \by^{(n)}(t)=y'^{(n)}+
t(\by^{(n)}(s^{(n)})-y'^{(n)})/\abs{\by^{(n)}(s^{(n)})-y'^{(n)}}$\,, for
$t\in[0,
\abs{\by^{(n)}(s^{(n)})-y'^{(n)}}]$ and 
$\tilde \by^{(n)}(t)= \by^{(n)}(t-\abs{\by^{(n)}(s^{(n)})-y'^{(n)}}+
s^{(n)})$\,, for
$t\ge
\abs{\by^{(n)}(s^{(n)})-y'^{(n)}}$\,.
The last time $\tilde\by^{(n)}$ visits $\cl\!B_\eta$ is
 $\tilde s^{(n)}=\abs{\by^{(n)}(s^{(n)})-y'^{(n)}}$\,. The
$\tilde s^{(n)}$ are thus bounded.
With $\tilde t^{(n)}=t^{(n)}-
s^{(n)}+\abs{\by^{(n)}(s^{(n)})-y'^{(n)}}$\,,
$\tilde t^{(n)}$ is the smallest $t$ with
$\tilde\by^{(n)}(t)\not\in D_\eta$\,.
 In addition, it is possible to choose $\eta$ small enough
to ensure that 
 $I^\ast_{y'^{(n)}}(\tilde\by^{(n)})\le
 I^\ast_{y'^{(n)}}(\by^{(n)})+\beta$\,.

Since $y'^{(n)}(t^{(n)})\notin D_\eta$\,,
 $Y^{(n)}(t^{(n)})\notin D$ provided 
$\abs{Y^{(n)}(t^{(n)})-\by^{(n)}(t^{(n)})}<\eta$ so that
 \begin{equation*}
   P_{y'^{(n)}}(\tau^{(n)}\le t^{(n)})\ge P_{y'^{(n)}}(\sup_{s\le t^{(n)}}\abs{Y^{(n)}(s)-\by^{(n)}(s)}<\eta)\,.
 \end{equation*}
Denote $D^{(n)}=D\cap\mathbb S^{(n)}_\Theta$ and
$B^{(n)}_r=B_r\cap\mathbb S^{(n)}_\Theta$\,. 
Assuming that $t^{(n)}\le T_2$\,, that $y'^{(n)}\in B^{(n)}_\eta$\,,
 that $y'^{(n)}\to y'$\,,  that
$t^{(n)}\to \hat t$  and that $\by^{(n)}\to\hat\by$ yield
\begin{multline*}
  \liminf_{n\to\infty}\frac{1}{n}\,\ln P_{y'^{(n)}}(\tau^{(n)}\le T_2)
\ge\liminf_{n\to\infty}\frac{1}{n}\,\ln  P_{y'^{(n)}}(\tau^{(n)}\le
  t^{(n)})\\
\ge \liminf_{n\to\infty}\frac{1}{n}\,\ln
P_{y'^{(n)}}(\sup_{s\le t^{(n)}}\abs{Y^{(n)}(s)-\by^{(n)}(s)}<\eta)
\ge-\inf\{I^\ast_{y',\hat t}( \by):\,\sup_{s\le
  \hat t}\abs{\by(s)-\hat\by(s)}<\eta\} \\
\ge -I^\ast_{y',\hat t}(\hat \by)\ge -I^\ast_{y'}(\hat \by)\ge
-U_\eta-2\beta
\ge -U-3\beta\,.
\end{multline*}
Since, for arbitrary $y'\in D^{(n)}$\,, 
$P_{y'}(\tau^{(n)}\le T_1+ T_2)\ge
P_{y'}(\sigma^{(n)}\le T_1)\inf_{y''\in
  B^{(n)}_\eta}P_{y''}(\tau^{(n)}\le T_2)$\,,  the argument of the proof
of \eqref{eq:23} yields
\begin{equation*}
  \liminf_{n\to\infty}\frac{1}{n}\,\ln\inf_{y'\in D^{(n)}}
 P_{y'}(\tau^{(n)}\le T_3)\ge -U\,,
\end{equation*}
where 
 $T_3=T_1+T_2$\,.
Thus, for  all $\beta>0$\,, 
\begin{equation}
  \label{eq:76}
\sup_{y'\in D^{(n)}}
  P_{y'}(\tau^{(n)}> T_3)\le 1-e^{n(-U-\beta)}\,,
\end{equation}
 provided $n$ is great enough.
By the Markov property, for $\ell \in \N$\,,
\begin{multline}
  \label{eq:81}
    P_{y'}(\tau^{(n)}>\ell T_3|\tau^{(n)}>(\ell-1)T_3)\\
=P_{y'}(Y^{(n)}(t)\in D\,, t\in[(\ell-1)T_3,\ell T_3]
|Y^{(n)}(t)\in D\,,t\in[0,(\ell-1)T_3])
\\=P_{y'}(Y^{(n)}(t)\in D\,, t\in[0, T_3]
|Y^{(n)}(0)\in D)\le \sup_{y''\in D^{(n)}}P_{y''}(\tau^{(n)}>T_3)\,,
\end{multline}
which implies that
\begin{equation*}
  \sup_{y'\in D^{(n)}}P_{y'}(\tau^{(n)}>\ell T_3)\le
\sup_{y'\in  D^{(n)}}P_{y'}(\tau^{(n)}> T_3)^\ell\,.
\end{equation*}
By \eqref{eq:76},
\begin{equation*}
  \sup_{y'\in D^{(n)}}  P_{y'}(\tau^{(n)}>\ell T_3)\le (1-e^{n(-U-\beta)})^\ell\,.
\end{equation*}
Therefore,
\begin{multline*}
    E_{y^{(n)}}(\tau^{(n)})^m=
m\int_0^\infty u^{m-1}P_{y^{(n)}}(\tau^{(n)}>u)\,du \le 
mT_3\sum_{\ell=0}^\infty((\ell+1)T_3)^{m-1}
P_{y^{(n)}}(\tau^{(n)}>\ell T_3)\\
\le mT_3^m\sum_{\ell=0}^\infty(\ell+1)^{m-1} (1-e^{n(-U-\beta)})^\ell
\le mT_3^m(1-e^{n(-U-\beta)})^{-2}
\int_0^\infty u^{m-1} (1-e^{n(-U-\beta)})^u\,du
\\
= \frac{m!T_3^m(1-e^{n(-U-\beta)})^{-2}}{(-\ln(1-e^{n(-U-\beta)}))^m}
\le m!T_3^m(1-e^{n(-U-\beta)})^{-2}e^{mn(U+\beta)}
\end{multline*}
proving \eqref{eq:9}.

Let $\sigma^{(n)}_0=0$ and, for $i\in \Z_+$\,,
\begin{align*}
  \tau^{(n)}_{i}=\inf\{t>\sigma^{(n)}_{i}:\,Y^{(n)}(t)\in B_{\eta}\}
\intertext{ and }
  \sigma^{(n)}_{i+1}=\inf\{t>\tau^{(n)}_{i}:\,Y^{(n)}(t)\not\in B_{2\eta}\}\,.
\end{align*}
One has that
\begin{multline}
  \label{eq:73}
  P_{y^{(n)}}(\tau^{(n)}<e^{n(U-3\kappa)})=\sum_{i=0}^\infty
  P_{y^{(n)}}(\sigma^{(n)}_{i}\le\tau^{(n)}<\tau^{(n)}_{i}\,,
  \tau^{(n)}<e^{n(U-3\kappa)})\\
\le \sum_{i=0}^{\lfloor e^{n(U-2\kappa)}\rfloor}
  P_{y^{(n)}}(\sigma^{(n)}_{i}\le\tau^{(n)}<\tau^{(n)}_{i})
+P_{y^{(n)}}(  \tau^{(n)}_{\lfloor
  e^{n(U-2\kappa)}\rfloor}<e^{n(U-3\kappa)})\\
=P_{y^{(n)}}(\tau^{(n)}<\tau^{(n)}_{0})+
 \sum_{i=1}^{\lfloor e^{n(U-2\kappa)}\rfloor}
  E_{y^{(n)}}(\ind_{\{\sigma^{(n)}_{i}\le\tau^{(n)}\}}P_{y^{(n)}}(
\tau^{(n)}<\tau^{(n)}_{i}|\mathcal{F}^{(n)}(\sigma^{(n)}_{i})))\\
+P_{y^{(n)}}\bl( \frac{ \tau^{(n)}_{\lfloor
    e^{n(U-2\kappa)}\rfloor}}{e^{n(U-2\kappa)}}
<e^{-n\kappa}\br)\,.
\end{multline}
Since 
$\tau^{(n)}=\inf\{t\ge \sigma^{(n)}_{i}:\,Y^{(n)}(t)\not\in D\}$
and $\tau_i^{(n)}=\inf\{t\ge \sigma^{(n)}_{i}:\,Y^{(n)}(t)\in
B_\eta\}$
on the event $\{\sigma^{(n)}_{i}\le\tau^{(n)}\}$\,,
by the strong Markov property, for $i\in\N$\,, $\ell\in \N$\,, and $n$
great enough,
\begin{multline}
  \label{eq:77}
  \ind_{\{\sigma^{(n)}_{i}\le\tau^{(n)}\}}
P_{y^{(n)}}(
\tau^{(n)}<\tau^{(n)}_{i}|\mathcal{F}^{(n)}(\sigma^{(n)}_{i}))=
  \ind_{\{\sigma^{(n)}_{i}\le\tau^{(n)}\}}P_{Y^{(n)}(\sigma^{(n)}_{i})}(
\tau^{(n)}<\tau^{(n)}_{0})\\
\le \sup_{y'\in B^{(n)}_{3\eta}}P_{y'}(
\tau^{(n)}<\tau^{(n)}_{0})
\le \sup_{y'\in B^{(n)}_{3\eta}}P_{y'}(
\tau^{(n)}_{0}>\ell T_1)+\sup_{y'\in B^{(n)}_{3\eta}}P_{y'}(\tau^{(n)}\le\ell T_1)\,.
\end{multline}
For arbitrary $\tilde y^{(n)}\in B^{(n)}_{3\eta}$ converging to some $\tilde y$\,, 
$P_{\tilde y^{(n)}}(\tau_0^{(n)}\le T_1)\to1 $\,, as $n\to\infty$\,. 
Moreover, with $\tilde \by$
solving \eqref{eq:27}  for $\tilde\by(0)=\tilde y$
and with $\tilde t_1$ being the length of   time it takes $\tilde\by$
to reach $\cl B_{\eta/2}$\,,
\begin{equation*}
    P_{\tilde y^{(n)}}(\tau_0^{(n)}> T_1)\le
P_{\tilde y^{(n)}}(\tau_0^{(n)}> \tilde t_1)\le
P_{\tilde y^{(n)}}(Y^{(n)}(\tilde t_1)\not\in B_{\eta}) 
\le P_{\tilde y^{(n)}}(\abs{Y^{(n)}(\tilde t_1)-\tilde \by(\tilde t_1)}>\frac{\eta}{2})
\end{equation*}
so that
\begin{equation}
  \label{eq:83}
  \limsup_{n\to\infty}\frac{1}{n}\,\ln
P_{\tilde y^{(n)}}(\tau_0^{(n)}> T_1)\le -\inf (I^\ast_{\tilde y}(\by'):\,
\abs{\by'(\tilde t_1)-\tilde \by(\tilde t_1)}\ge\frac{\eta}{2}
)<0\,.
\end{equation}
It follows that, for some $\chi>0$ and  all $n$ great enough,
\begin{equation*}
  \sup_{y'\in B^{(n)}_{3\eta}}
P_{y'}(\tau_0^{(n)}> T_1)\le e^{-n\chi}\,.
\end{equation*}
For $\ell\in\N$\,, in analogy with \eqref{eq:81},
\begin{equation*}
 \sup_{y'\in B_{3\eta}^{(n)}} P_{y'}(\tau_0^{(n)}> \ell T_1|\tau_1^{(n)}> (\ell-1) T_1)
\le\sup_{y'\in B_{3\eta}^{(n)}}  P_{y'}(\tau_0^{(n)}> T_1)
\end{equation*}
so that
\begin{equation}
  \label{eq:87}
  \sup_{y'\in B^{(n)}_{3\eta}} P_{y'}(\tau_0^{(n)}> \ell T_1)\le e^{-n\ell \chi}\,.
\end{equation}
The second term on the rightmost side  of \eqref{eq:77} is dealt with next.
Let $\breve y^{(n)}\in B^{(n)}_{3\eta}$ converge to $\breve y\in \cl B_{3\eta}$\,.
Then, for some $\breve\by\in\{\by:\,\by(0)=
\breve y\,,\by(t)\not \in D\text{ for some
}t\in[0,\ell T_1]\}$\,, the latter set being closed and denoted by $F$\,,
\begin{multline}
  \label{eq:82}
  \limsup_{n\to\infty}
\frac{1}{n}\,\ln P_{\breve y^{(n)}}(\tau^{(n)}\le\ell T_1)=
  \limsup_{n\to\infty}
\frac{1}{n}\,\ln P_{\breve y^{(n)}}(Y^{(n)}(t)\not \in D\text{ for some
}t\le\ell T_1)\\
\le-\inf_{\by\in F}I^\ast_{\breve y}(\by)=-I^\ast_{\breve y}(\breve\by)\,.
\end{multline}
Now, if $\check\by$ with $\check \by(0)=\nu(\rho)$
is obtained from $\breve\by$ by inserting a straight line segment joining
points $\nu(\rho)$ and $\breve\by(0)$\,, then, provided $\eta$ is
small enough, $I^\ast_{\nu(\rho)}(\check\by)\le I^\ast_{\breve
  y}(\breve\by)+\kappa/2$ so that, for all $n$ great enough,
\begin{equation}
  \label{eq:3}
      \limsup_{n\to\infty}
\frac{1}{n}\,\ln P_{\breve y^{(n)}}(\tau^{(n)}\le\ell T_1)
  \le
-\inf_{y' \notin D}\Phi(\nu(\rho),y')+\frac{\kappa}{2}\,.
\end{equation}
 Thus, for all $n$ great enough,
\begin{equation}
  \label{eq:88}
  \sup_{y'\in B^{(n)}_{3\eta}}P_{y'}(\tau^{(n)}\le\ell T_1)
\le e^{-n(U-\kappa)}\,.
\end{equation}
 By \eqref{eq:77}, \eqref{eq:87} and \eqref{eq:88}, for $i\in\N$\,,
 choosing $\ell$
judiciously, for $n$ great,
\begin{equation}
  \label{eq:89}
  \ind_{\{\sigma^{(n)}_{i}\le\tau^{(n)}\}}    P_{y^{(n)}}(
\tau^{(n)}<\tau^{(n)}_{i}|\mathcal{F}^{(n)}(\sigma^{(n)}_{i}))\le 2e^{-n(U-\kappa)}\,.
\end{equation}
The first term on the rightmost side of \eqref{eq:73} is tackled similarly.
In analogy with \eqref{eq:83}, 
\begin{equation*}
    \limsup_{n\to\infty}\frac{1}{n}\,\ln
P_{y^{(n)}}(\tau_0^{(n)}>T_1)\le -\inf (I^\ast_{ y}(\by'):\,
\abs{\by'(\overline t_1)-\overline \by(\overline t_1)}\ge\frac{\eta}{2}
)<0\,,
\end{equation*}
where $\overline \by$
solves \eqref{eq:27}  with $\overline\by(0)= y$ and
$\overline t_1$ is the length of time  it takes $\overline\by$ to hit
$\cl B_{\eta/2}$\,.
In analogy with \eqref{eq:82} and \eqref{eq:3},
with $F'=\{\by:\,\by(t)\not \in  D\text{ for some
}t\le T_1\}$\,,
\begin{equation*}
    \limsup_{n\to\infty}
\frac{1}{n}\,\ln P_{y^{(n)}}(\tau^{(n)}\le T_1)=
  \limsup_{n\to\infty}
\frac{1}{n}\,\ln P_{y^{(n)}}(Y^{(n)}(t)\not \in D\text{ for some
}t\le T_1)
\le-\inf_{\by'\in F'}I^\ast_{ y}(\by')\,.
\end{equation*}
The set $F'$ being closed, the latter infimum is attained. On the
other hand, the solution of \eqref{eq:27} started at $y$ does not
belong to $F'$ as it does not leave $D$\,, so, the infimum is less
than zero.
Hence, $\chi>0$ can be chosen to satisfy the inequality, for all $n$
great enough,
\begin{equation*}
    P_{y^{(n)}}(\tau_0^{(n)}> T_1)\vee P_{y^{(n)}}(\tau^{(n)}\le T_1)\le e^{-n\chi}\,.
\end{equation*}
Therefore,
\begin{equation}
  \label{eq:94}
    P_{y^{(n)}}(
\tau^{(n)}<\tau^{(n)}_{0})
\le P_{y^{(n)}}(
\tau^{(n)}_{0}> T_1)+P_{y^{(n)}}(\tau^{(n)}\le T_1)\le 2e^{-n\chi}\,.
\end{equation}

By \eqref{eq:73},  \eqref{eq:89}, and \eqref{eq:94},
\begin{equation}
  \label{eq:90}
    P_{y^{(n)}}(\tau^{(n)}<e^{n(U-3\kappa)})\le
4 e^{-n\kappa}+P_{y^{(n)}}\bl(e^{-n(U-2\kappa)} \tau^{(n)}_{\lfloor
    e^{n(U-2\kappa)}\rfloor})
<e^{-n\kappa}\br)\,.
\end{equation}
Since
\begin{equation*}
  \tau^{(n)}_{\lfloor
    e^{n(U-2\kappa)}\rfloor}\ge\sum_{i=1}^{\lfloor
    e^{n(U-2\kappa)}\rfloor}(\tau^{(n)}_i-\sigma^{(n)}_{i})\wedge1\,,
\end{equation*}
by the strong Markov property,
$E_{y^{(n)}}((\tau^{(n)}_i-\sigma^{(n)}_{i})\wedge1|\mathcal{F}^{(n)}(\sigma^{(n)}_{i}))
=E_{Y^{(n)}(\sigma^{(n)}_{i})}(\tau^{(n)}_0\wedge1)\ge
\inf_{y'\in B^{(n)}_{3\eta}\setminus B^{(n)}_{2\eta}}E_{y'}(\tau^{(n)}_0\wedge1)$\,,
so, assuming $n$ is great enough,
 \begin{multline*}
  P_{y^{(n)}}\bl(e^{-n(U-2\kappa)}  \tau^{(n)}_{\lfloor
    e^{n(U-2\kappa)}\rfloor}\le e^{-n\kappa}\br)\le
P_{y^{(n)}}\bl( e^{-n(U-2\kappa)} \sum_{i=1}^{\lfloor
    e^{n(U-2\kappa)}\rfloor}(\tau^{(n)}_i-\sigma^{(n)}_{i})\wedge1\le
  e^{-n\kappa}\br)
\\=
P_{y^{(n)}}\bl( e^{-n(U-2\kappa)} \sum_{i=1}^{\lfloor
    e^{n(U-2\kappa)}\rfloor}\bl((\tau^{(n)}_i-\sigma^{(n)}_{i})\wedge1
-E_{y^{(n)}}((\tau^{(n)}_i-\sigma^{(n)}_{i})\wedge1)\br)\\\le e^{-n\kappa}-
e^{-n(U-2\kappa)}\sum_{i=1}^{\lfloor
    e^{n(U-2\kappa)}\rfloor}E_{y^{(n)}}((\tau^{(n)}_i-\sigma^{(n)}_{i})\wedge1)\br)
\\\le 
P_{y^{(n)}}\bl( e^{-n(U-2\kappa)} \sum_{i=1}^{\lfloor
    e^{n(U-2\kappa)}\rfloor}\bl((\tau^{(n)}_i-\sigma^{(n)}_{i})\wedge1
-E_{y^{(n)}}((\tau^{(n)}_i-\sigma^{(n)}_{i})\wedge1)\br)\\\le e^{-n\kappa}-
\frac{1}{2}\inf_{y'\in B^{(n)}_{3\eta}\setminus B^{(n)}_{2\eta}}E_{y'}(\tau^{(n)}_0\wedge1)
\br)\,.
\end{multline*}
Let $T'>0$ denote the infimum of time lengths it takes a solution of
\eqref{eq:27} to get from a point in $B_{3\eta}\setminus B_{2\eta}$ to a point in
$B_{\eta}$\,. Since $Y^{(n)}$ started at point $y'^{(n)}\in
B_{3\eta}^{(n)}\setminus B_{2\eta}^{(n)}$ such that $y'^{(n)}\to y'$ converges to the solution of \eqref{eq:27}
started at $y'$ locally uniformly, $\tau^{(n)}_0$ is greater than
$T'/2$ with great $P_{y'^{(n)}}$--probability\,, 
for $n$ great enough. Therefore,
$\liminf_{n\to\infty}E_{y'^{(n)}}(\tau^{(n)}_0\wedge1)\ge (T'/2)\wedge1$\,.
Hence, $e^{-n\kappa}-
(1/2)\inf_{y'\in B^{(n)}_{3\eta}\setminus B^{(n)}_{2\eta}}E_{y'}(\tau^{(n)}_0\wedge1)<0$\,, for all
$n$ great enough, so that, on recalling that
 the $\tau^{(n)}_i-\sigma^{(n)}_{i}$\,, for $i\in\N\,,$ are
independent, by Chebyshev's inequality,
\begin{multline*}
  P_{y^{(n)}}\bl( e^{-n(U-2\kappa)} \sum_{i=1}^{\lfloor
    e^{n(U-2\kappa)}\rfloor}\bl((\tau^{(n)}_i-\sigma^{(n)}_{i})\wedge1
-E_{y^{(n)}}((\tau^{(n)}_i-\sigma^{(n)}_{i})\wedge1)\br)\le e^{-n\kappa}-
\frac{1}{2}\,\inf_{y'\in B^{(n)}_{3\eta}\setminus B^{(n)}_{2\eta}}E_{y'}(\tau^{(n)}_0\wedge1)
\br)\\\le
P_{y^{(n)}}\bl( e^{-n(U-2\kappa)}\abs{ \sum_{i=1}^{\lfloor
    e^{n(U-2\kappa)}\rfloor}\bl((\tau^{(n)}_i-\sigma^{(n)}_{i})\wedge1
-E_{y^{(n)}}((\tau^{(n)}_i-\sigma^{(n)}_{i})\wedge1)\br)}\ge \frac{(T'/2)\wedge1}{2}- e^{-n\kappa}
\br)\\\le \frac{e^{-2n(U-2\kappa)}}{(((T'/2)\wedge1)/2- e^{-n\kappa})^2}\,
\sum_{i=1}^{\lfloor
    e^{n(U-2\kappa)}\rfloor}\text{Var}_{y^{(n)}}((\tau^{(n)}_i-\sigma^{(n)}_{i})\wedge1)
\le \frac{e^{-n(U-2\kappa)}}{(((T'/2)\wedge1)/2- e^{-n\kappa})^2}\,.
\end{multline*}
Thus, the righthand side of \eqref{eq:90} converges to $0$ so that
\eqref{eq:9'} has been proved.

\def\cprime{$'$} \def\cprime{$'$} \def\cprime{$'$} \def\cprime{$'$}
  \def\cprime{$'$} \def\polhk#1{\setbox0=\hbox{#1}{\ooalign{\hidewidth
  \lower1.5ex\hbox{`}\hidewidth\crcr\unhbox0}}} \def\cprime{$'$}
  \def\cprime{$'$} \def\cprime{$'$} \def\cprime{$'$} \def\cprime{$'$}
  \def\cprime{$'$}

\end{document}